\documentclass[a4paper,11pt]{amsart}
\usepackage[left=2.5cm, right=2.5cm,top=2.5cm,bottom=3cm]{geometry}

\usepackage[utf8]{inputenc}
\usepackage[english]{babel}
\usepackage{amssymb}
\usepackage{amsfonts}
\usepackage{dsfont}

\usepackage{hyperref}

\usepackage{tikz}

\usepackage{graphicx}
\graphicspath{ {./FIGURES/} }
\usepackage{hhline}
\usepackage[sort&compress,comma,square,numbers]{natbib}

\DeclareMathOperator{\GL}{GL}
\DeclareMathOperator{\Log}{Log}

\DeclareMathOperator{\N}{N}

\DeclareMathOperator{\Conv}{Conv}
\DeclareMathOperator{\Cone}{Cone}
\DeclareMathOperator{\Vertex}{Vert}
\newcommand{\R}{\mathbb{R}}

\newtheorem{thm}{Theorem}[section] 
\newtheorem{cor}[thm]{Corollary} 
\newtheorem{prop}[thm]{Proposition} 
\newtheorem{lemma}[thm]{Lemma}
\newtheorem{problem}[thm]{Problem}

\theoremstyle{definition}
\newtheorem{definition}[thm]{Definition}
\newtheorem{ex}[thm]{Example}
\newtheorem{remark}[thm]{Remark}

\begin{document}

\title{On generalizing Descartes' rule of signs to hypersurfaces}

\author{Elisenda Feliu}
\address{Department of Mathematical Sciences, University of Copenhagen,
Universitetsparken 5,
2100 Copenhagen, Denmark}
\email{efeliu@math.ku.dk}

\author{Máté L. Telek}
\address{Department of Mathematical Sciences, University of Copenhagen,
Universitetsparken 5,
2100 Copenhagen, Denmark}
\email{mlt@math.ku.dk}

\begin{abstract}
We give partial generalizations of the classical Descartes' rule of signs to multivariate polynomials (with real exponents), in the sense that we provide upper bounds on the number of connected components of the complement of a hypersurface in the positive orthant. In particular, we give conditions based on the geometrical configuration of the exponents and the sign of the coefficients that guarantee that the number of connected components where the polynomial  attains a negative value is at most one or two. 
Our results fully cover the cases where such an  upper bound provided by the univariate Descartes' rule of signs   is one. This approach opens a new route to generalize  Descartes' rule of signs to the multivariate case, differing from previous works that aim at counting the number of positive solutions of a system of multivariate polynomial equations.

\medskip
\emph{Keywords: } semi-algebraic set; signomial; Newton polytope; connectivity; convex function
\end{abstract}

\maketitle

\section{Introduction}
Descartes’ rule of signs, established by René Descartes in his book {\it La Géométrie}  in 1637, provides an easily computable upper bound for the number of positive real roots of a univariate polynomial with real coefficients.
Specifically, 
it states that the polynomial  cannot have more positive real roots than the number of sign changes in its coefficient sequence (excluding zero coefficients). In 1828, Gauss improved the rule by showing that the number of positive real roots, counted with multiplicity, and the number of sign changes in the coefficients sequence, have the same parity \cite{Gauss1828}. 
Since then, several different proofs were published e.g. \cite{Curtiss1918, Albert1943, GenDescartes}, and several generalizations were made in several directions. In 1918, Curtiss gave a proof that works for real exponents and even for some infinite series  \cite{Curtiss1918}. In 1999, Grabiner  showed that Descartes' bound is sharp, that is, for every given sign sequence, one can always find compatible coefficients   such that the polynomial has the maximum possible number of positive roots provided by Descartes’ bound \cite{Grabiner_DescartesIsSharp}. Generalizations of the Descartes’ rule to other types of functions in one variable are also available \cite{Haukkanen_Tossavainen,Tokarev2011}.
 
 Efforts to generalize Descartes’ rule of signs to the multivariate case have focused   on systems of $n$ multivariate polynomial equations  in $n$ variables, and on bounding the number of solutions  in the positive orthant using sign properties of the coefficients of the system. The first conjecture for such a bound was published in 1996 by Itenberg and Roy \cite{Itenberg_Roy}. They were able to show their conjecture for some special cases. The first non-trivial example supporting the conjecture was presented by  Lagarias and Richardson \cite{Lagarias1997} in 1997. Almost at the same time, Li and Wang gave a counterexample to the Itenberg-Roy conjecture \cite{Li_Wang}.  The first  generalization was given recently and identifies systems with at most one solution in the positive orthant   \cite{MullerSigns}, see also \cite{CGS}. Afterwards,  a sharp upper bound was given for systems of polynomials supported on circuits  \cite{Bihan_2016, bihan2020optimal}. 
In these works, the bound is given in terms of the sign variation of a sequence associated both with the exponents and the coefficients of the system. To the best of our knowledge, these are the only known generalizations of Descartes’ rule of signs to the multivariate case.
 
Descartes' rule of signs allows however for a ``dual'' presentation: it gives an upper bound on the number of connected components of $\R_{>0}$ minus the zero set of the polynomial, and if the sign of the highest degree term is fixed, then it also gives an upper bound on the number of  connected components where the polynomial evaluates positively or negatively. 
Specifically, 
if we write $f(x)= a_0+a_1x + \dots + a_n x^n$ with $a_n\neq 0$, and let $\rho$ be the Descartes' bound on the number of positive roots, then there are at most $\rho+1$ connected components. If $\rho$ is odd, the upper bounds for the number of components where $f$ is positive or negative agree, while if $\rho$ is even, then there are at most $\tfrac{\rho}{2}+1$    connected components where $f$ attains the sign of $a_n$.
For example, if after ignoring zero coefficients, the sign sequence of the coefficients  is 
$(++--)$, then there is one connected component where the polynomial evaluates positively and one where it evaluates negatively. If the  sequence is 
$(+- + -)$, then there at most two connected components where the polynomial evaluates positively and at most two where it evaluates negatively, see Fig.~\ref{FIG0e}.
 
 With this presentation, Descartes' rule of signs may be generalized to hypersurfaces in the following sense. Let $f\colon \R^n_{>0} \rightarrow \R$ be a \emph{signomial} (a multivariate generalized polynomial, where we allow real exponents, restricted to the positive orthant), and consider the sets 
\begin{equation}\label{eq:V}
V_{>0}(f) := \{ x \in \mathbb{R}^{n}_{>0} \mid f(x) = 0 \}, \qquad V_{>0}^{c}(f) := \mathbb{R}^{n}_{>0} \setminus V_{>0}(f).
\end{equation}
We aim at bounding the number of connected components of  $V_{>0}^{c}(f) $ in terms of the relative position of the exponent vectors of each monomial of $f$ in $\R^n$, and the sign of the coefficients. 
This leads to the formulation of the following problem for the \emph{generalization  of Descartes' rule of signs to hypersurfaces.}

\begin{problem}
\label{TheProblem}
Consider a signomial $f\colon \mathbb{R}^{n}_{>0} \to \mathbb{R}$ with  $f(x) = \sum_{\mu \in \sigma(f)}c_{\mu}x^{\mu}$,  and $\sigma(f)\subseteq \R^n$   a finite set. 
Find a (sharp) upper bound on the number of connected components of $V_{>0}^{c}(f)$, where $f$ takes  negative (resp. positive) values, based on the sign of the coefficients and the geometry of $\sigma(f)$. 
\end{problem}

\begin{figure}[t]
  \centering
  \begin{minipage}[b]{0.2\linewidth}
    \centering
    \includegraphics[scale=0.25]{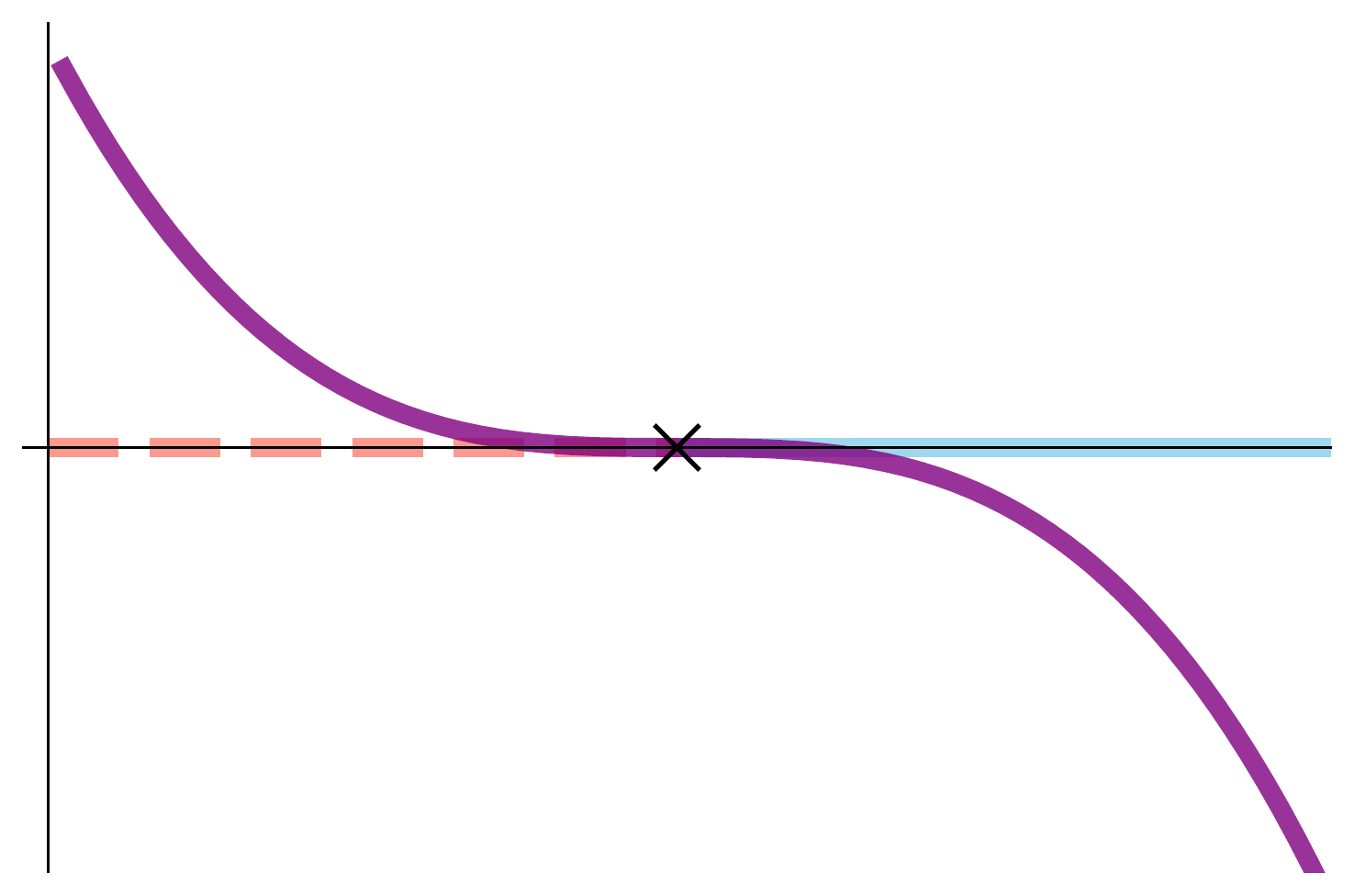}
    
    {\small (a)}
  \end{minipage}\qquad
  \begin{minipage}[b]{0.2\linewidth}
   \centering
    \includegraphics[scale=0.25]{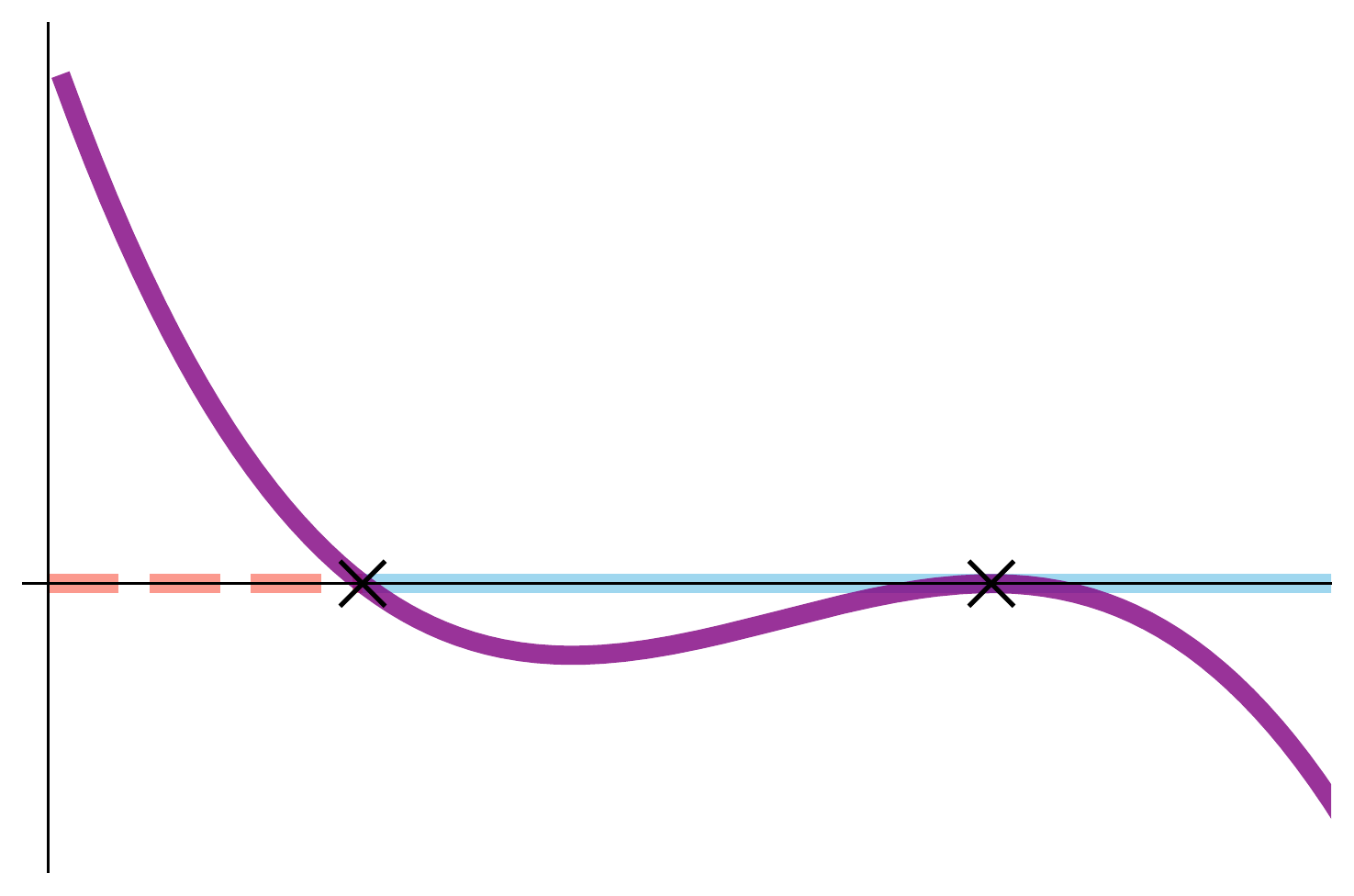}
    
        {\small (b)}
  \end{minipage}\qquad
  \begin{minipage}[b]{0.2\linewidth}
   \centering
    \includegraphics[scale=0.25]{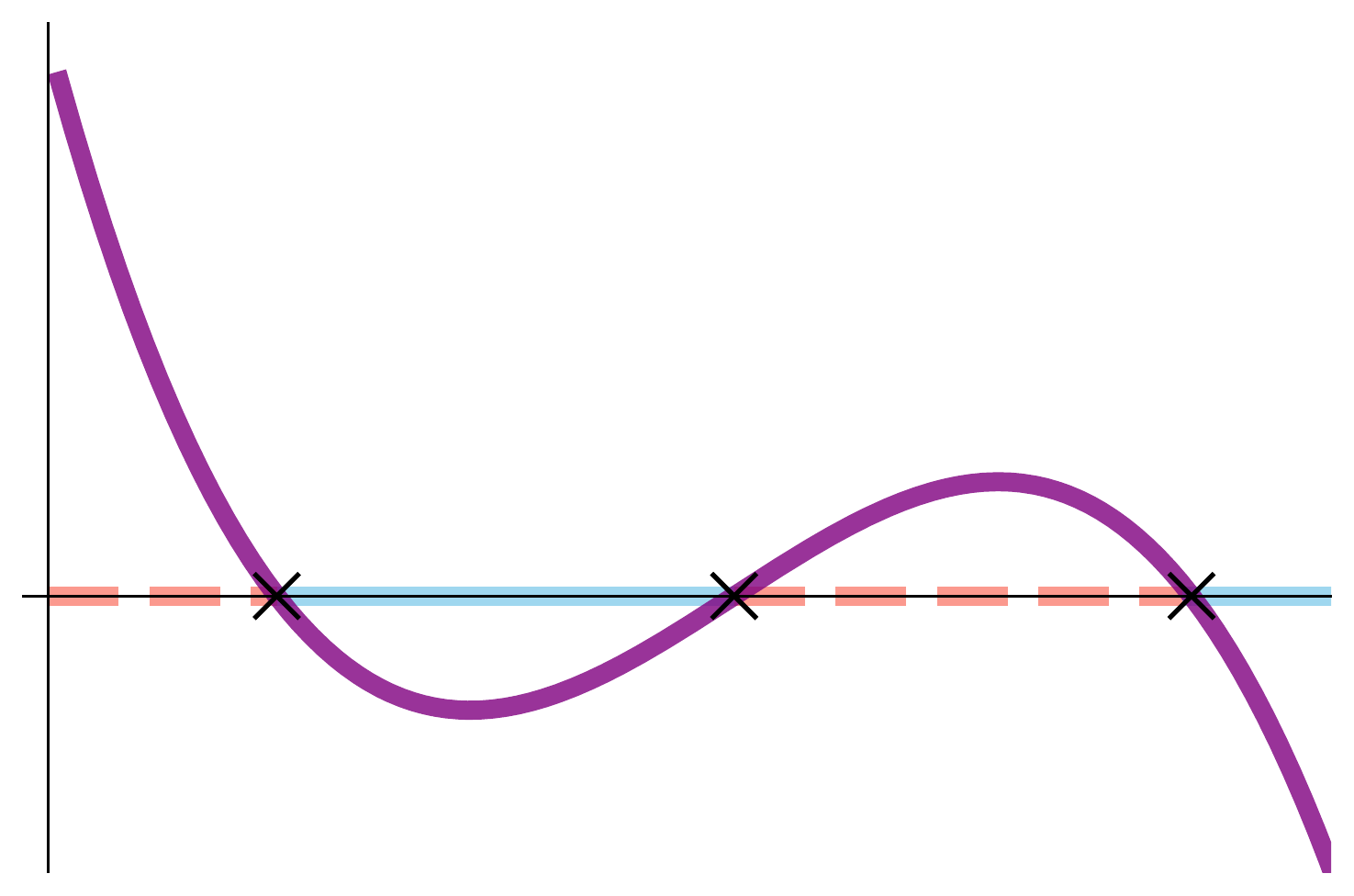}
    
        {\small (c)}
  \end{minipage}\qquad
  \begin{minipage}[b]{0.22\linewidth}
   \centering
    \includegraphics[scale=0.25]{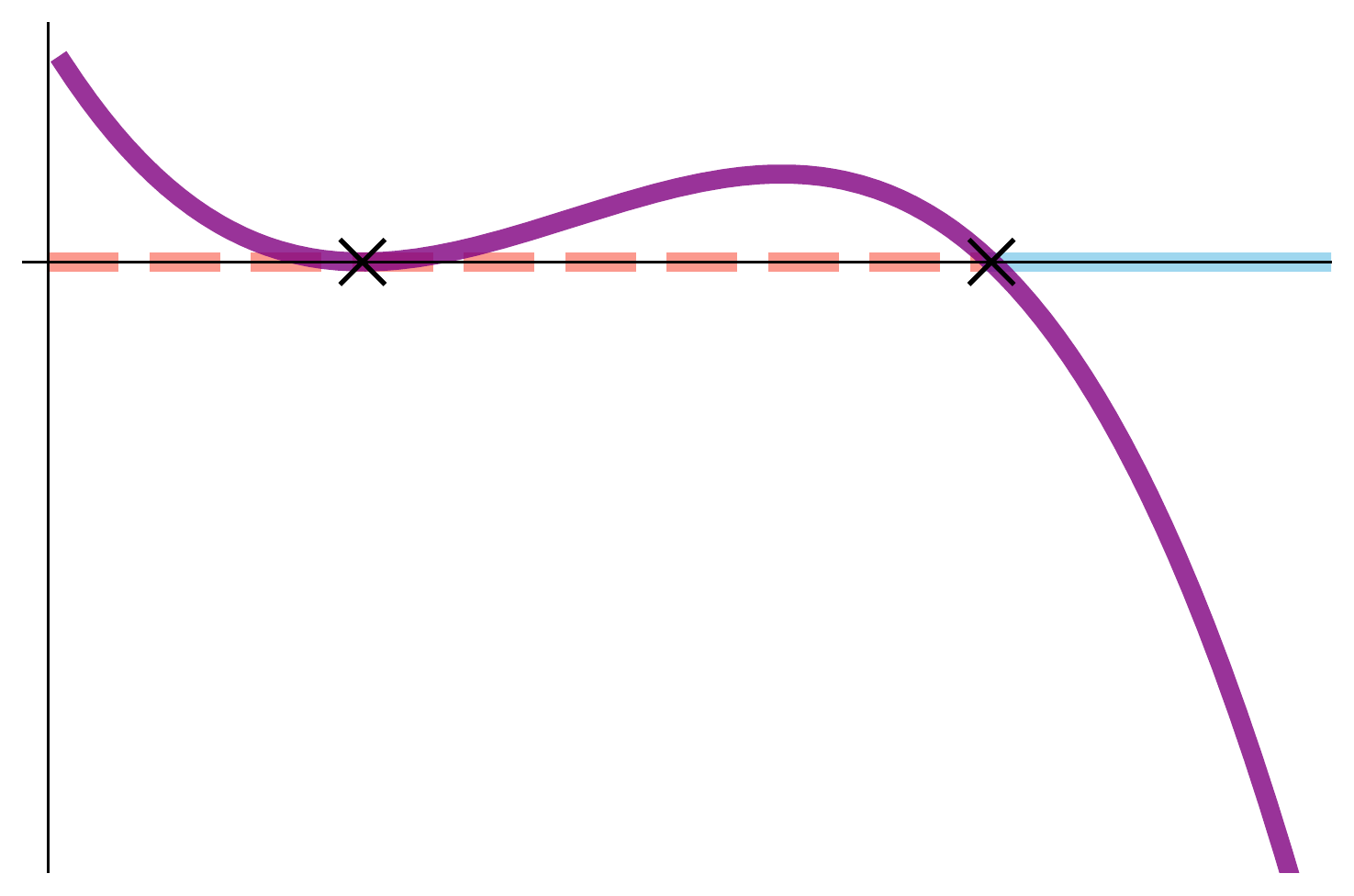}
    
        {\small (d)}
  \end{minipage}
\caption{{\small  Graphs of polynomials $p$ of degree three with coefficient sign sequence $(+-+-)$. In each figure, the connected components of $\R_{>0}$ minus the zero set of $p$, where $p$ evaluates positively or negatively, are shown in red and  blue respectively. 
 (a) $8-12x+6x^2-x^{3}$. (b) $9-15x + 7x^2 -x^{3}$. (c)  $ 15- 23x + 9x^{2} -x^{3} $. (d) $   3- 7x + 5x^{2}-x^{3}$.}}
\label{FIG0e}
\end{figure}

In this paper we address Problem~\ref{TheProblem} for generic $n$ in some scenarios, which, in particular, include the univariate Descartes' rule of signs when the upper bound on the number of connected components where $f$ is negative is one, that is, when   the sign sequence is one of $(+ \dots + - \dots -)$, \ $ (- \dots - + \dots +)$, or $(+ \dots + - \dots - + \dots +)$. 

Specifically, we show that $V_{>0}^{c}(f)$ has at most one connected component where $f$ is negative if $f$ has only one negative coefficient (Theorem~\ref{Thm_OneExponent}). The same holds if there exists a hyperplane
separating the exponents with positive coefficients from those with negative coefficients  (Theorem~\ref{Thm_StrictSepConnected}), or if the exponents with negative coefficient lie on a simplex such that the exponents with positive coefficient lie outside the simplex in a certain way (Theorem~\ref{Thm_Convexification}). A detailed account of our results is given in Section~\ref{Section_Discussion}. 
We focus on finding upper bounds for the number of negative connected components, as statements about the number of positive connected components of $V_{>0}^{c}(f)$ follows by studying $-f$. 

\medskip
If $f$ is a polynomial, that is, $\sigma(f) \subseteq \mathbb{Z}_{\geq 0}^{n}$, the set $V_{>0}^{c}(f)$ is semi-algebraic and hence it has a finite number of connected components \cite[Theorem 5.22]{Basu_book}.  Computing topological invariants of semi-algebraic sets, such as the number of connected components, has been heavily studied in real algebraic geometry. Upper bounds of the sum of the Betti numbers of a semi-algebraic set in terms of the number of variables, the degree and the number of the defining polynomials can be found for example in \cite[Theorem 1]{Basu_BoundingtheBetti}, \cite[Theorem 6.2]{Gabrielov_Vorobjov}, and \cite[Theorems 1.8 and 2.7]{Basu_Multi-degree}. For the  number of connected components of a semi-algebraic set, that is, the $0$-th Betti number,  an upper bound was given in \cite[Theorem 1]{Basu_OnTheNumberOfCells}, \cite[Theorem 1.1]{Basu_RefinedBounds}.

There exist several algorithms to compute the number of connected components of a semi-algebraic set. One algorithm is provided by Cylindrical Algebraic Decomposition, but it has double exponential complexity (see \cite[Remark 11.19]{Basu_book}). A more efficient way to compute connected components is using so-called road maps. In this way, one has an algorithm with single exponential complexity. For more details about this algorithm, see \cite[Section 3]{Basu_Survey}.  

\medskip
The Descartes' rule of signs is of special importance in applications where positive solutions to polynomial systems are the object of study. This is the case in models in biology and (bio)chemistry where variables are  concentrations or abundances. It is precisely in this setting, namely  the theory of biochemical reaction networks, where our motivation to consider Problem \ref{TheProblem} comes from. 
In an upcoming paper, we show that the connectivity of the set of parameters that give rise to  multistationarity in a reaction network \cite{craciun-feinberg-pnas,FeliuPlos,joshi-shiu-III} relies on the number of connected components of the complementary of a hypersurface. The hypersurface of interest is  large for realistic networks, with many monomials and variables, and hence not manageable by algorithms from  semi-algebraic geometry. The advantage of the techniques presented here is that they rely on linear optimization problems, and can handle this application.

The paper is organized as follows. In Section~\ref{Preliminaries}, we provide the notation and basic results on signomials. In Section~\ref{Section_Path_on_log_paper}, we give bounds answering  Problem \ref{TheProblem} using   separating  hyperplanes (Theorem \ref{Thm_StrictSepConnected}, \ref{Thm_StripAroundPosExp}), while  in Section~\ref{Convexification}   bounds  are found by providing conditions that guarantee that the signomial can be transformed into a convex function, while preserving the number of connected components of  $V_{>0}^{c}(f)$ (Theorem~\ref{Thm_Convexification}). In Section  \ref{Section_Discussion}, we compare the two approaches.
Throughout we illustrate our results with examples and figures, worked out using SageMath \cite{sagemath}.

\subsection*{Notation} $\mathbb{R}_{\geq0}$, $\mathbb{R}_{>0}$ and $\R_{<0}$ refer to the sets of non-negative, positive and negative real numbers respectively. We denote the Euclidean scalar product of two vectors $v,w \in \mathbb{R}^{n}$ by $v \cdot w$. For a set $\sigma \subseteq \mathbb{R}^{m}$, a matrix $M \in \mathbb{R}^{n \times m}$ and a vector $v \in \mathbb{R}^{n}$ we write $M\sigma + v$ for the set $\{Ms + v \mid s \in \sigma \}$. For two sets $A,B \subseteq \mathbb{R}^{n}$, the set $A+B=\{ a + b \mid a \in A, \, b \in B \}$ is the Minkowski sum of $A$ and $B$. We let $\Conv(A)$ denote the convex hull of $A$. For $a_1,\dots,a_m\in \R^n$, we  write $\Conv(a_1,\dots,a_m):=\Conv(\{a_1,\dots,a_m\})$. 
By convention, the maximum over an empty set is $-\infty$, and the minimum over an empty set is $\infty$. The symbol $\#S$ denotes the cardinality of a finite set $S$.

\section{Preliminaries}
\label{Preliminaries}
The central object of study is a function
\begin{align}\label{eq:sig}
f\colon \mathbb{R}^{n}_{>0} \to \mathbb{R},  \qquad f(x)= \sum\limits_{\mu \in \sigma(f)}c_{\mu}x^{\mu}, \quad \textrm{with }c_\mu \in \R \setminus \{0\},
\end{align}
where $\sigma(f) \subseteq \mathbb{R}^{n}$ is a finite set, called the \textit{support} of $f$. Here $x^{\mu}$ is the usual short notation for $x_{1}^{\mu_{1}}  \dots x_{n}^{\mu_{n}}$. To emphasize that we restrict the domain of $f$ to the positive orthant, we call $f$ a \textit{signomial}. That is, a signomial is a generalized polynomial on the positive orthant.  The term signomial was introduced by Duffin and Peterson in the early 1970s  \cite{duffin1973geometric}. Since then, it is commonly used in geometric programming \cite{GeomProgramming, FuzzyGeo}. 

Given a signomial $f$ as in \eqref{eq:sig} and a set $S\subseteq \sigma(f)$, we define \emph{the restriction of $f$ to $S$} by considering the monomials with exponent vectors in $S$: 
\begin{equation}\label{eq:res} 
f_{|S}(x)=  \sum\limits_{\mu \in S}c_{\mu}x^{\mu}.
\end{equation}

With the notation in \eqref{eq:V} and by continuity,  the signomial $f$ has constant sign in each connected component of  $V_{>0}^{c}(f)$.

\newpage
\begin{definition} Let $f$ be a signomial in $n$ variables. 
\begin{itemize}
\item A connected component $U$ of  $V_{>0}^{c}(f)$ is said to be \textit{positive} if $f(x) > 0$ for every $x \in U$. We say $U$ is \textit{negative}, if $f(x) < 0$ for every $x \in U$.
\item  The convex hull of   $\sigma(f)$ is called the \textit{Newton polytope} of $f$ and denoted by $\N(f)$. 
\item A point  $\alpha \in \sigma(f)$ is called positive, resp. negative, if the coefficient $c_\alpha$ is positive, resp. negative. The set  $\sigma(f)$ is partitioned into the set of positive points and the set of negative points: 
\begin{align*}
\sigma_{+}(f) := \{ \alpha \in \sigma(f) \mid c_\alpha > 0 \}\quad \text{        and        } \quad \sigma_{-}(f) := \{ \beta \in \sigma(f)  \mid c_{\beta} < 0 \}.
\end{align*}
\end{itemize}
\end{definition}

\begin{figure}[t]
\centering
\begin{minipage}[h]{0.4\textwidth}
\centering
\includegraphics[scale=0.5]{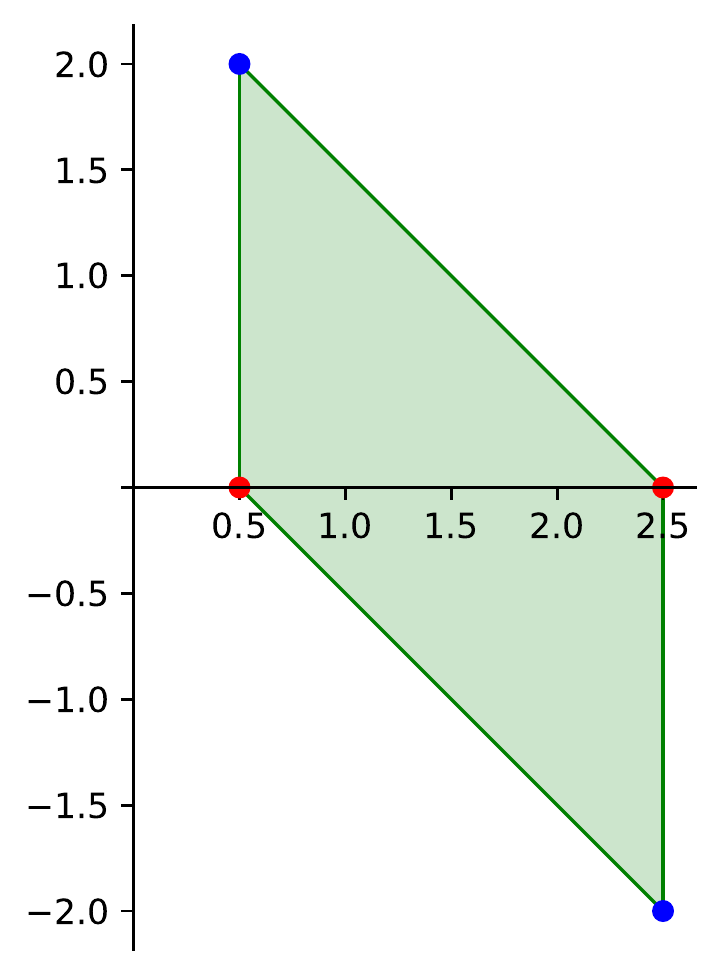}

{\small (a)}
\end{minipage}
\begin{minipage}[h]{0.4\textwidth}
\centering
\includegraphics[scale=0.5]{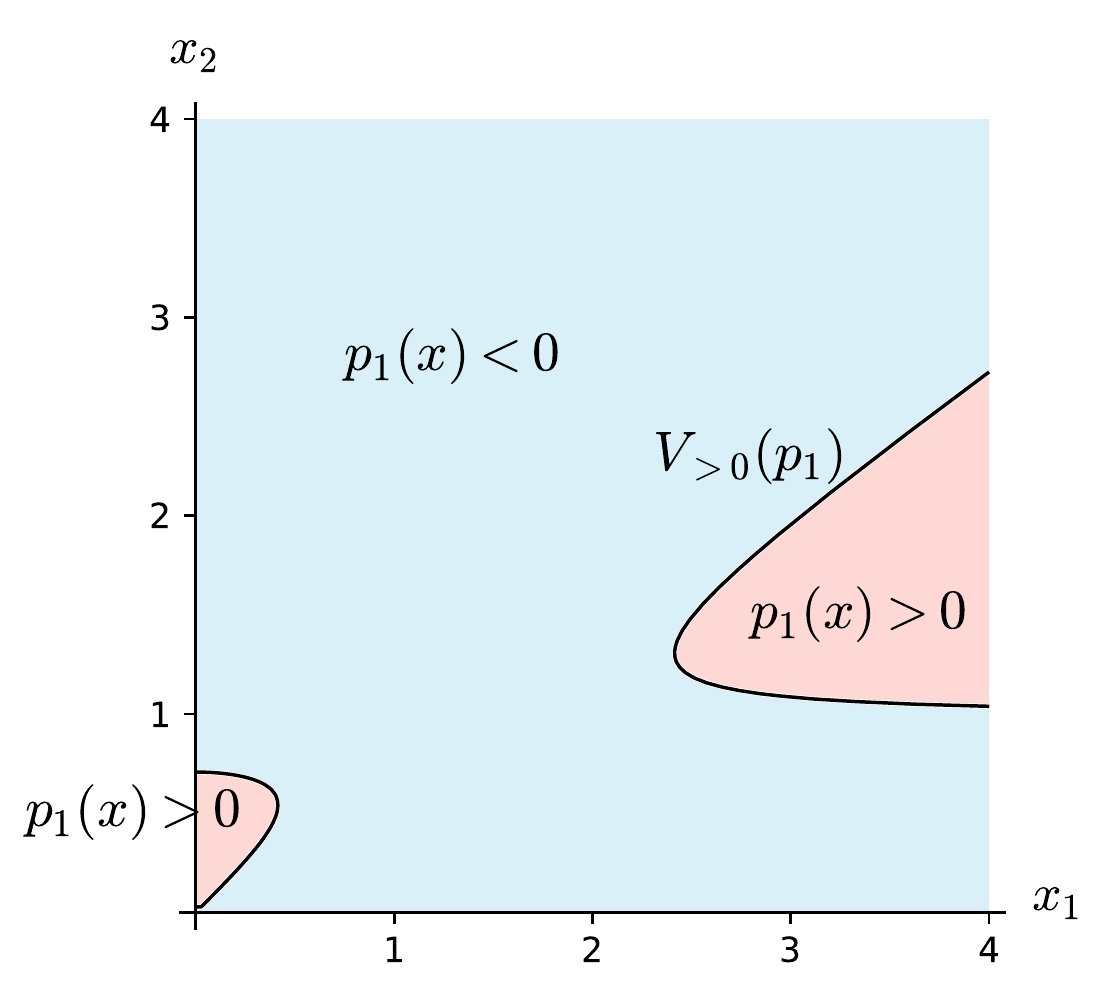}

{\small (b)}
\end{minipage}
\caption{{\small (a) Newton polytope of $p_{1}(x_{1},x_{2})$ from Example~\ref{Example_2_1}. Blue points are negative  and red points are positive. (b) The positive and negative connected components of $V_{>0}^{c}(p_{1})$.}}\label{FIG_Example_2_1}
\end{figure}

\begin{ex}
\label{Example_2_1}
The  support of the signomial 
\[ p_{1}(x_{1},x_{2}) = x_{1}^{2.5} - 2x_{1}^{0.5}x_{2}^{2} + x_{1}^{0.5} - x_{1}^{2.5}x_{2}^{-2}\]
 is  $\sigma(p_1) = \{ (2.5,0), (0.5,2) , (0.5,0), (2.5,-2) \}$. The points $(2.5,0)$, $(0.5,0)$ are positive, while the points $(0.5,2)$, $(2.5,-2)$ are negative. The Newton polytope of $p_{1}$ and the positive and negative connected components of $V_{>0}^c(p_{1})$ are displayed in Fig.~\ref{FIG_Example_2_1}.
\end{ex}

In what follows, it will be convenient to 
consider transformations of the support that do not change the number of negative (resp. positive) connected components. Any invertible matrix $M \in \GL_{n}(\mathbb{R})$ induces a function
\begin{align}
\label{Nota_MonChange}
h_{M} \colon \mathbb{R}^{n}_{>0} \to \mathbb{R}^{n}_{>0}, \qquad  x \mapsto x^M:=(x^{M_{1}}, \dots ,x^{M_{n}})
\end{align} 
where $M_{1}, \dots , M_{n}$ denote the columns of $M$. The function $h_{M}$ is called a  \textit{monomial change of variables} and it is a homeomorphism. 
\begin{lemma}
\label{Lemma_TransInvariant}
For $M \in \GL_{n}(\mathbb{R})$, $v \in \mathbb{R}^{n}$, and a signomial $f$ on $\mathbb{R}^{n}_{>0}$, define the signomial
\[ F_{M,v,f}\colon \mathbb{R}^{n}_{>0} \to \mathbb{R}, \qquad F_{M,v,f}(x)= x^v f(h_M(x)). \]   
There is a homeomorphism between the positive (resp. negative) connected components of $V_{>0}^c(f)$ and $V_{>0}^c(F_{M,v,f})$. Furthermore,
\begin{align*}
\sigma_{+}\big(F_{M,v,f}\big) = M \sigma_{+}(f)+v \quad \text{        and        } \quad \sigma_{-}\big(F_{M,v,f}\big) = M \sigma_{-}(f)+v.
\end{align*}
\end{lemma}

\begin{proof}
If $f(x)=\sum\nolimits_{\mu \in \sigma(f)}c_{\mu}x^{\mu}$, we have
\[F_{M,v,f}(x)=  x^v f(h_M(x)) = \sum\limits_{\mu \in \sigma(f)} c_{\mu} x^v (x^M)^{\mu} = 
 \sum\limits_{\mu \in \sigma(f)} c_{\mu}  x^{M\mu+ v}. \]
From this, the second part of the lemma follows.

For the first part, clearly, the identity map induces a sign-preserving homeomorphism between 
$V_{>0}^c(F_{M,v,f})$ and $V_{>0}^c(f\circ h_M)$, and the map $h_M$ induces a homeomorphism between $V_{>0}^c(f\circ h_M)$ and $V_{>0}^c(f)$, which also preserves the sign of each connected component.  
\end{proof}

In view of Lemma~\ref{Lemma_TransInvariant}, we can for example assume that all exponent vectors belong to $\R^n_{>0}$ if necessary.  Moreover, if $\sigma(f) \subseteq \mathbb{Q}^{n}$, then $f$ can be replaced by a polynomial and the number of negative (resp. positive) connected components of $V_{>0}^{c}(f)$ remains unchanged.

\begin{ex}
The matrix $M = \begin{pmatrix}
 0.5 &  0.5 \\
  0.5 &  0
 \end{pmatrix}$ and the vector $v = (-0.25,-0.25)$ transform the signomial $p_{1}$ from Example \ref{Example_2_1} to the polynomial $F_{M,v,p_1}(x_{1},x_{2}) = x_{1}x_{2} - 2 x_{2}+1-x_{1}$.
\end{ex}

\section{Paths on logarithmic scale}
\label{Section_Path_on_log_paper}
In this section, we provide the first results towards Problem \ref{TheProblem}. The idea behind the proofs in this section relies on reducing 
the multivariate signomial to a univariate signomial, and applying Descartes' rule of signs. 
To this end, given $v \in \mathbb{R}^{n}$ and $x \in \mathbb{R}^{n}_{>0}$, we consider continuous paths
\begin{align}\label{Nota_Path}
\gamma_{v,x}\colon [1,\infty) \to \mathbb{R}^{n}_{>0}, \qquad t \mapsto (t^{v_{1}}x_{1},\dots ,t^{v_{n}}x_{n}).
\end{align}
In logarithmic scale, applying the coordinate-wise natural logarithm map  
\begin{align}
\label{Eq_Log}
 \Log\colon \mathbb{R}^{n}_{>0} \to \mathbb{R}^{n}, \qquad (x_{1},\dots,x_{n}) \mapsto (\log(x_{1}),\dots,\log(x_{n})),
\end{align}
each path $\gamma_{v,x}$ is transformed into a half-line $\tau_{v,\Log(x)}\colon 
[0,\infty) \to \mathbb{R}^{n}$, $s \mapsto s\, v+\Log(x)$, 
with start point $\Log(x)$ and direction vector $v$.  Specifically, 
\begin{equation}\label{eq:Log}
 \Log \circ \, \gamma_{v,x}  = \tau_{v,\Log(x)}  \circ \log, \qquad \textrm{in }[1,\infty).
\end{equation}%
Since the logarithm map $\Log$ is a homeomorphism, the topological properties of $f^{-1}( \R_{<0})$ and of its image under $\Log$ are the same. This observation gives us an easy geometric way to think about paths $ \gamma_{v,x}$.

Given a signomial $f$, each   $v \in \mathbb{R}^{n}$ and $x \in \mathbb{R}^{n}_{>0}$ induce a signomial function in one variable:
\begin{align}\label{Nota_IndFunc}
f_{v,x}\colon \mathbb{R}_{>0} \to \mathbb{R}, \qquad t \mapsto \sum\limits_{\mu \in \sigma(f)} (c_{\mu}x^{\mu}) t^{v\cdot \mu}. 
\end{align}
Note that $f_{v,x}(1)=f(x)$. 
Since the restriction of $f_{v,x}$ to $[1,\infty)$ is the composition $f \circ \gamma_{v,x}$, understanding the properties of $f_{v,x}$ allows us to determine whether the path $\gamma_{v,x}$ is in the pre-image of the negative real line under $f$. This motivates the study of signomials in one variable.
The following lemma will be used repeatedly in what follows. Its proof is a direct application of Descartes' rule of signs.

\begin{lemma}
\label{Lemma_OneDimSignomial}
Let $g\colon \mathbb{R}_{>0} \to \mathbb{R}$, $g(t)=\sum_{\nu \in \sigma(g)}a_{\nu}t^{\nu}$, be a signomial in one variable such that $g(1)<0$. 
\begin{itemize}
\item[(i)] If the sign sequence of the coefficients of $g$ has at most two sign changes, and the leading coefficient is positive, then there is unique $\rho \in (1,\infty )$ such that $g(\rho)=0$, and it holds that $g(t) < 0$ for all $t\in [1,\rho)$ and $g(t) > 0$ for all $t\in (\rho,\infty)$.
\item[(ii)] If the sign sequence of the coefficients of $g$ has at most one sign change,  and the leading coefficient is negative, then $g(t) < 0$ for all $t\in [1,\infty)$.
\end{itemize}
\end{lemma}

Following the notation of \cite[Section 2.3.1]{joswig2013polyhedral} and \cite[Section 1.1]{grunbaum2003convex}, for
every $v \in \mathbb{R}^{n}$ and $a \in \mathbb{R}$, we define a \textit{hyperplane} $\mathcal{H}_{v,a} := \{ \mu \in \mathbb{R}^{n} \mid v \cdot \mu = a \}$,  
and two \textit{half-spaces}
\begin{align*}
\mathcal{H}_{v,a}^{+} := \{ \mu \in \mathbb{R}^{n} \mid  v \cdot \mu \geq a \}\quad \text{        and        } \quad\mathcal{H}_{v,a}^{-} := \{ \mu \in \mathbb{R}^{n} \mid  v \cdot \mu \leq a \}.
\end{align*}
We let $\mathcal{H}_{v,a}^{+,\circ},\mathcal{H}_{v,a}^{-,\circ}$ denote the interior of $\mathcal{H}_{v,a}^{+}$, and $\mathcal{H}_{v,a}^{-}$ respectively. 
Although $\mathcal{H}_{v,a} =\mathcal{H}_{-v,-a}$, the choice of sign determines which half-space is positive  and which one is negative. 

As we will see in Lemma \ref{Lemma_BehavOffvx}, the relative position of a hyperplane $\mathcal{H}_{v,a}$ and the points in $\sigma(f)$ gives valuable information about the behavior of the function $f_{v,x}$ in \eqref{Nota_IndFunc}. 
To this end, we introduce the following types of vectors $v$.

\begin{definition}
\label{Def_Strip}
Let $v\in \R^n$. 
\begin{enumerate}
\item[(i)] We say that $v$ is a \textit{separating vector} of $\sigma(f)$ if for some $a\in \R$ it holds
\[ \sigma_-(f) \subseteq \mathcal{H}_{v,a}^{+}, \qquad \sigma_+(f) \subseteq \mathcal{H}_{v,a}^{-}.\]
The separating vector $v$ is \emph{strict} if  $\sigma_-(f) \cap \mathcal{H}_{v,a}^{+,\circ}\neq \emptyset$, and  \emph{very strict} if additionally $\sigma_{-}(f)\cap \mathcal{H}_{v,a} = \emptyset$ for some $a\in \R$. Let $\mathcal{S}^{-}(f)$ denote the set of separating vectors of $\sigma(f)$. 

\item[(ii)] We say that $v$ is an \emph{enclosing vector} of $\sigma(f)$ if for some $a,b\in \R$, $a\leq b$, it holds
\[  \sigma_-(f) \subseteq \mathcal{H}_{v,a}^{+} \cap \mathcal{H}_{v,b}^{-}, \qquad \sigma_+(f) \subseteq \R^n\setminus 
(\mathcal{H}_{v,a}^{+,\circ} \cap \mathcal{H}_{v,b}^{-,\circ}). 
\] 
We say that  $v$ is a \emph{strict enclosing vector} of $\sigma(f)$ if additionally 
$\sigma_+(f) \cap \mathcal{H}_{v,a}^{-,\circ}\neq \emptyset$ and $\sigma_+(f) \cap \mathcal{H}_{v,b}^{-,\circ}\neq \emptyset$. 
We denote by $\mathcal{E}^{-}(f)$ the set of enclosing vectors of $\sigma(f)$. 
\end{enumerate}
\end{definition}

The sets of separating and enclosing vectors can be described algebraically as
\begin{align} 
\mathcal{S}^{-}(f) &= \Big\{ v \in \mathbb{R}^{n} \mid   \max_{\alpha\in \sigma_{+}(f)}  v \cdot \alpha \leq \min\limits_{\beta \in \sigma_{-}(f)} v \cdot \beta \Big\}, \label{eq:Salg} \\ 
\mathcal{E}^{-}(f) &= \{ v \in \mathbb{R}^{n} \mid \forall \alpha \in \sigma_{+}(f)\colon  v \cdot \alpha \leq \min\limits_{\beta \in \sigma_{-}(f)} v \cdot \beta \text{  or  } \max\limits_{\beta \in \sigma_{-}(f)} v \cdot \beta \leq v \cdot \alpha \}. \label{eq:Ealg}
\end{align}

 For $v\in \mathcal{S}^{-}(f) $, setting $a:= \max_{\alpha\in \sigma_{+}(f)}  v \cdot \alpha$, Definition~\ref{Def_Strip}(i) holds. For $v\in \mathcal{E}^{-}(f)$, 
we let $a:=\min_{\beta \in \sigma_{-}(f)} v \cdot \beta$ and $b:=\max_{\beta \in \sigma_{-}(f)} v \cdot \beta$
and Definition~\ref{Def_Strip}(ii) holds.
 
Note that a separating vector is in particular an enclosing vector, that is, $\mathcal{S}^{-}(f) \subseteq \mathcal{E}^{-}(f)$. 
Using the algebraic description of $\mathcal{S}^-(f)$ from  \eqref{eq:Salg}, one can easily show that $\mathcal{S}^{-}(f)$ is a convex cone, i.e. it is closed under addition and multiplication by a nonnegative scalar \cite[Ch. 1]{Ziegler_book}. 

For a separating vector $v$ to be strict, there must be a negative point in $\sigma(f)$ in $\mathcal{H}_{v,a}^{+}$ that is not in the hyperplane $\mathcal{H}_{v,a}$. 
That is,  there exists $\beta_{0} \in \sigma_{-}(f)$ such that $\max_{\alpha\in \sigma_{+}(f)}  v \cdot \alpha  < v \cdot \beta_0$. For it to be very strict, no negative point of   $\sigma(f)$ lies on the hyperplane, or equivalently, the inequality defining $\mathcal{S}^{-}(f) $ in \eqref{eq:Salg} is strict. 
Fig.~\ref{FIG3}(a) shows a strict separating vector.  

Enclosing vectors  \emph{enclose} all negative points of  $\sigma(f)$ between two parallel hyperplanes separated from the positive points, but points of both signs are allowed to be in the two hyperplanes.
For an enclosing vector $v$ to be strict, there must be positive points on the side of the hyperplanes not containing the negative points, that is, there exist $\alpha_1,\alpha_2\in \sigma_+(f)$ such that the inequalities in \eqref{eq:Ealg} are strict for that $v$ respectively.  See Fig.~\ref{FIG4}(a).

Enclosing and separating vectors order the exponents of $f_{v,x}$ in \eqref{Nota_IndFunc}, such that the negative and positive coefficients are grouped. This has the following consequences.

\begin{lemma}
\label{Lemma_BehavOffvx}
Let $f\colon \mathbb{R}^{n}_{>0} \to \mathbb{R}$ be a signomial and $x \in \R^n_{>0}$.
\begin{itemize}
\item[(i)] If $v \in \mathcal{E}^{-}(f)$, then there are  at most two sign changes in the coefficient sign sequence of the signomial $f_{v,x}$.  If $v$ is additionally strict, then both  the leading coefficient and the coefficient of smallest degree of $f_{v,x}$ are positive. 
\item[(ii)] If $v \in \mathcal{S}^{-}(f)$, then there is at most one sign change in the coefficient sign sequence of the signomial $f_{v,x}$.  If $v$ is strict, then the leading coefficient of $f_{v,x}$ is negative. 
\end{itemize}

Additionally if $f(x)<0$, then the following statements hold:
\begin{itemize}
\item[(i')] If $v \in \mathcal{E}^{-}(f)$, then there is a unique $\rho \in (1,\infty]$ such that $f_{v,x}(t) < 0$ for all $t\in [1,\rho)$ and $f_{v,x}(t) > 0$ for all $t > \rho$ (note that $\rho$ might be $\infty$).
\item[(ii'')] If $v \in \mathcal{S}^{-}(f)$, then    $f_{v,x}(t) < 0$ for all $t\in [1,\infty)$.
\end{itemize}

\end{lemma}

\begin{proof}
(i) and (i'). 
For $v \in \mathcal{E}^{-}(f)$, 
 $v$ orders the exponents $v\cdot \mu$ such that the sign sequence is $(+\dots+ - \dots - + \dots +)$, with potentially one or more of the three blocks of repeated signs not present.  The positive blocks are present if $v$ is strict by definition, showing (i).

For  $f(x)<0$, if the leading coefficient of $f_{v,x}$ is positive, then Lemma \ref{Lemma_OneDimSignomial}(i) gives the existence of a   unique $\rho \in (1,\infty)$ satisfying (i') in the statement. 
If the leading coefficient of $f_{v,x}$ is negative, then $v\in \mathcal{S}^{-}(f)$ and this case is covered next, and gives $\rho=\infty$.

\smallskip
(ii) and (ii'). From $v \in \mathcal{S}^{-}(f)$, it follows that the signomial $f_{v,x}$ has at most one sign change in its coefficient sequence,  as $ \max_{\alpha\in \sigma_{+}(f)}  v \cdot \alpha \leq \min_{\beta \in \sigma_{-}(f)} v \cdot \beta$. 
If $v$ is strict, then for at least one $\beta_{0} \in \sigma_{-}(f)$ we have $\max_{\alpha\in \sigma_{+}(f)}  v \cdot \alpha <  v \cdot \beta_{0}$, and hence the leading term is negative, showing (ii).
If $f_{v,x}(1) =f(x) < 0$,  $f_{v,x}$ must have some negative coefficient. Using $v \in \mathcal{S}^{-}(f)$, we conclude that the leading coefficient is negative and $v$ is strict. Lemma~\ref{Lemma_OneDimSignomial}(ii) gives now statement (ii').
\end{proof}

\begin{thm}
\label{Thm_OneExponent}
Let $f\colon \mathbb{R}^{n}_{>0} \to \mathbb{R}$ be a signomial. If at most one coefficient of $f$ is negative,  
then $f^{-1}(\R_{<0})$ is a logarithmically convex set. In particular,  $V_{>0}^c(f)$ has at most one negative connected component. 
\end{thm}
\begin{proof}
Let $x, y \in f^{-1}(\R_{<0})$, define $v := \Log(y) - \Log(x)$, and let $e$ denote Euler's number.
Since $f$ has at most one negative coefficient,  $v$ is an enclosing vector, c.f. Definition~\ref{Def_Strip}(ii). Since $f_{v,x}(1) =  f(x)  < 0$ and $f_{v,x}(e) = f(y) < 0$, Lemma \ref{Lemma_BehavOffvx}(i') implies that $f_{v,x}(t) < 0$ for all $t\in [1,e]$ and 
hence $\gamma_{v,x}(t)\in f^{-1}( \R_{<0})$ for $t\in [1,e]$. Applying $\Log$, equality \eqref{eq:Log} gives that 
$\tau_{v,\Log(x)}(s) \in \Log( f^{-1}( \R_{<0}) )$ for all $s\in [0,1]$. As $\tau_{v,\Log(x)}$ in the interval $[0,1]$ is simply the line segment joining $\Log(x)$ and $\Log(y)$, $\Log( f^{-1}( \R_{<0}) )$ is convex. This concludes the proof. 
\end{proof}

We will now show that the existence of one strict separating vector implies that $V^c_{>0}(f)$ has at most one negative connected component, which in addition is contractible. To this end, we need an auxiliary proposition, that states that the existence of 
one very strict separating vector is enough to guarantee that there is a basis of very strict separating vectors. The idea is simply that the property of being a very strict separating vector is robust under small perturbations.

For a finite collection of vectors $w_{1}, \dots ,w_{k} \in \mathbb{R}^{n}$ we write
\begin{align}
\label{Eq_DefCone}
\Cone(w_{1}, \dots, w_{k}) := \Big\{ \sum_{i=1}^{k} \lambda_{i} w_{i} \mid \lambda_{1}, \dots , \lambda_{k} \in \mathbb{R}_{\geq 0} \Big\}
\end{align} 
for the convex cone generated by $w_{1}, \dots ,w_{k}$. If $w_{1}, \dots ,w_{k}$ are linearly independent, then the relative interior of $\Cone(w_{1}, \dots w_{k})$ is given by
\begin{align}
\label{Eq_IntCone}
\Cone^{\circ}(w_{1}, \dots, w_{k}) = \Big\{ \sum_{i=1}^{k} \lambda_{i} w_{i} \mid \lambda_{1}, \dots , \lambda_{k} \in \mathbb{R}_{> 0} \Big\}.
\end{align}

\begin{prop}\label{prop_basis}
Let $f\colon \mathbb{R}^{n}_{>0} \to \mathbb{R}$ be a signomial and $v \in \mathbb{R}^{n}$ a very strict separating vector of $\sigma(f)$. 
Then there exists a basis $\{w_1,\dots,w_n\}$ of $\R^n$ consisting of very strict separating vectors,  and a constant $c \in \mathbb{R}$ such that 
\begin{align}
\label{eq:prop_basis}
&\sigma_-(f) \subseteq \mathcal{H}_{w_i,c}^{+}, \qquad  \sigma_+(f) \subseteq \mathcal{H}_{w_i,c}^{-} \qquad \text{for every} \quad  i \in \{1,\dots,n\}, \\
\label{eq:prop_intcone}
& v \in \Cone^{\circ}(w_{1}, \dots ,w_{n}).
\end{align}
\end{prop}
\color{black}

\begin{proof}
Define 
\[  a:= \max\limits_{\alpha \in \sigma_{+}(f)}\  v \cdot \alpha, \qquad b:= \min\limits_{\beta \in \sigma_{-}(f)} v \cdot \beta, \qquad c := \tfrac{a+b}{2}. \]
As $v\in \mathcal{S}^-(f)$, $\sigma_-(f) \subseteq \mathcal{H}_{v,c}^{+}$ and $\sigma_+(f) \subseteq \mathcal{H}_{v,c}^{-}$ by \eqref{eq:Salg}. Since $v$ is very strict, we have $b> c >a$.

Choose a basis $\{v_{1}, \dots ,v_{n} \}$ of $\mathbb{R}^{n}$ such that $v \in \Cone^{\circ}(v_{1},\dots ,v_{n})$. By \eqref{Eq_IntCone}  this is equivalent to the existence of $\lambda_{1}, \dots ,\lambda_{n} \in \mathbb{R}_{>0}$ such that $v = \sum_{i=1}^{n} \lambda_{i} v_{i}$. 
For this basis, we define
\[ K := \min_{i=1,\dots,n}\ \min\limits_{\mu \in \sigma(f)} v_{i} \cdot \mu, \qquad L := \max_{i=1,\dots,n}\ \max\limits_{\mu \in \sigma(f)} v_{i} \cdot \mu.\]

In the following, we show that it is possible to choose $\epsilon_{i} >0$ such that the vectors $w_{i} := v + \epsilon_{i} \, v_{i}$, for $i = 1, \dots , n$, with the given $c$ satisfy \eqref{eq:prop_basis}. 
For  $\beta \in \sigma_{-}(f)$ and  $i \in \{ 1, \dots , n\}$, using that $v_{i} \cdot \beta \geq K$ and $v \cdot \beta\geq b$, it holds that 
\begin{align}\label{eq:beta} w_{i} \cdot \beta = v \cdot \beta + \epsilon_{i} \, (v_{i} \cdot \beta) \geq b + \epsilon_{i} K  \  \begin{cases} 
\geq b > c & \textrm{if } K \geq 0 \text{ and for } \epsilon_{i} >0, \\
> b +  \tfrac{a-b}{2K} K = c  & \textrm{if } K<0 \text{ and for } 0 < \epsilon_{i} < \tfrac{a-b}{2K}.
\end{cases} \end{align}
Similarly, for every $\alpha \in \sigma_{+}(f)$ and  $i \in \{1, \dots , n\}$, it follows that 
\begin{align}\label{eq:alpha} w_{i} \cdot \alpha = v \cdot \alpha + \epsilon_{i}\, (v_{i} \cdot \alpha) \leq a + \epsilon_{i} L  \  \begin{cases} 
\leq a < c & \textrm{if } L \leq 0 \text{ and for } \epsilon_{i}>0, \\
< a +  \tfrac{b-a}{2L} L = c  & \textrm{if } L>0 \text{ and for } 0 < \epsilon_{i} <\tfrac{b-a}{2L}.
\end{cases} \end{align}

Therefore, there exists an $\epsilon >0$ such that $w_{i}$ satisfies \eqref{eq:beta} and \eqref{eq:alpha} for all $0 < \epsilon_{i} < \epsilon$ and $i \in \{1, \dots , n\}$. Hence for sufficiently small $\epsilon_{1}, \dots ,\epsilon_{n}$ the vectors $w_{1}, \dots , w_{n}$ are very strict separating vectors satisfying \eqref{eq:prop_basis}.

To obtain \eqref{eq:prop_intcone}, we specify a choice of $\epsilon_{1}, \dots ,\epsilon_{n}$. For each $i \in \{ 1, \dots , n\}$, choose $p_{i} >0$ such that $\epsilon_{i} := \tfrac{\lambda_{i}}{p_{i}} < \epsilon$ and define $P := \sum_{i=1}^{n} p_{i}$. By construction, we have that
\[ \sum_{i = 1}^{n} \tfrac{p_{i}}{P+1} w_{i} = \sum_{i = 1}^{n} \tfrac{p_{i}}{P+1}(v+\tfrac{\lambda_{i}}{p_{i}}v_{i}) = \tfrac{P}{P+1} v + \tfrac{1}{P+1} \sum_{i = 1}^{n} \lambda_{i}v_{i} = v, \]
which gives that $v\in \Cone^{\circ}(w_{1}, \dots , w_{n}).$

Finally, since $v$ is a positive linear combination of $v_{1}, \dots ,v_{n}$ and $\epsilon_{1}, \dots, \epsilon_{n}$ are positive, an easy linear algebra argument shows that $w_{1}, \dots ,w_{n}$ form a basis of $\mathbb{R}^{n}$.
\end{proof}

\begin{thm}
\label{Thm_StrictSepConnected}
Let $f\colon \mathbb{R}^{n}_{>0} \to \mathbb{R}$ be a signomial. If there exists a strict separating vector of $\sigma(f)$, then 
\begin{itemize}
\item[(i)] $f^{-1}(\R_{<0})$  is non-empty and contractible.
\item[(ii)]  The closure of $f^{-1}(\R_{<0})$ equals $f^{-1}(\R_{\leq 0})$. 
\end{itemize}
In particular,  $V_{>0}^c(f)$ has at most one negative connected component. 
\end{thm}
\begin{proof} 
Let $v \in \mathcal{S}^{-}(f)$ be a strict separating vector. Define 
\[  a:= \max\limits_{\alpha \in \sigma_{+}(f)}\  v \cdot \alpha,  \quad\textrm{and}\quad
M := \{ \beta \in \sigma_{-}(f) \mid v \cdot \beta = a\} = \sigma_-(f) \cap \mathcal{H}_{v,a}.
\]
Since $v$ is a strict separating vector, 
$\sigma_{-}(f) \setminus M\neq \emptyset$. Consider the restriction of $f$ to $\sigma(f) \setminus M$, c.f. \eqref{eq:res}: 
\begin{align*}
\tilde{f} := f_{|\sigma(f) \setminus M}.
\end{align*}
As $\tilde{f}$ is obtained from $f$ only by removing monomials with negative coefficients,   $f(x) \leq \tilde{f}(x)$ for all $x \in \mathbb{R}^{n}_{>0}$ and hence $\tilde{f}^{-1}(\R_{<0}) \subseteq f^{-1}(\R_{<0})$.
By construction 
$\sigma_-(\tilde{f})\neq \emptyset$, and
 $v$ is also a strict separating vector of $\sigma(\tilde{f})$, which additionally satisfies
\[   \max\limits_{\alpha \in \sigma_{+}(\tilde{f})}\  v \cdot \alpha < \min\limits_{\beta \in \sigma_{-}(\tilde{f})} v \cdot \beta. \]
Hence, $v$ is a very strict separating vector of $\sigma(\tilde{f})$. 
Note that for any $x\in \mathbb{R}^{n}_{>0}$, 
 the leading coefficient of 
$\tilde{f}_{v,x}$ is negative  by Lemma~\ref{Lemma_BehavOffvx}(ii), and hence $\tilde{f}^{-1}(\R_{<0}) \neq  \emptyset$. It follows that $f^{-1}(\R_{<0}) \neq  \emptyset$ as well.

We show that $f^{-1}(\R_{<0})$  is  contractible, by showing that this is the case for $\Log(f^{-1}(\R_{<0}))$.
First, note that by Proposition~\ref{prop_basis}, there exists a basis $\{w_1,\dots,w_n\}$ of $\R^n$, consisting of very strict separating vectors of $\sigma(\tilde{f})$ such that $v$ can be written as 
\begin{equation}\label{eq:vcone}
v = \sum_{i=1}^{n} \lambda_{i} w_{i} \qquad \text{for some} \quad \lambda = (\lambda_{1}, \dots , \lambda_{n}) \in \mathbb{R}_{>0}^{n}.\end{equation}
To show that $\Log(f^{-1}(\R_{<0}))$ is contractible, we will show that for any $\xi   \in \Log(\tilde{f}^{-1}(\mathbb{R}_{<0}))$, it holds  that $\xi + \Cone(w_{1}, \dots, w_{n})$ is a strong deformation retract of $\Log(f^{-1}(\mathbb{R}_{<0}))$. 
As $\xi + \Cone(w_{1}, \dots w_{n})$ is contractible, this will conclude the proof of (i), c.f. \cite{Hatcher_book}.

To this end, fix $x\in\tilde{f}^{-1}(\mathbb{R}_{<0})$ and let $\xi = \Log(x)$. 
For $w\in \mathcal{S}^{-}(\tilde{f})$, the path $\gamma_{w,x}$ is contained in $\tilde{f}^{-1}( \mathbb{R}_{<0})$ by Lemma~\ref{Lemma_BehavOffvx}(ii'). Hence, by equality \eqref{eq:Log}, the path $\tau_{w,\xi}$ is contained in $\Log(\tilde{f}^{-1}(\mathbb{R}_{<0}))$. 
 In particular, it holds that  $\xi+w \in \Log( \tilde{f}^{-1}( \mathbb{R}_{<0}) )$ for all $w \in \mathcal{S}^{-}(\tilde{f})$.
As $\mathcal{S}^{-}(\tilde{f})$ is a convex cone and contains $w_{1}, \dots , w_{n}$, we have $\Cone(w_{1}, \dots, w_{n})\subseteq \mathcal{S}^{-}(\tilde{f})$ \cite[Ch. 1]{Ziegler_book}. It follows that $\xi + \Cone(w_{1}, \dots, w_{n}) \subseteq \Log(\tilde{f}^{-1}(\mathbb{R}_{<0}))\subseteq \Log(f^{-1}(\mathbb{R}_{<0}))$.

\smallskip
We now construct a homotopy map giving that $\xi + \Cone(w_{1}, \dots, w_{n})$ is a strong deformation retract of 
$ \Log(f^{-1}(\mathbb{R}_{<0}))$. 
To this end, we consider the map $s^*\colon \mathbb{R}^{n} \to \mathbb{R}_{\geq 0}$ defined by
\[ s^*(\zeta)=\min \{s\in \mathbb{R}_{\geq 0} \mid \zeta + s\, v \ \in \,   \xi +  \Cone(w_{1}, \dots , w_{n}) \}. \] 
To see that $s^*$ is well defined and continuous, we note that 
\[  s^*(\zeta) =   \max \Big\{0,-\tfrac{ (W^{-1}(\zeta- \xi))_1}{\lambda_{1}},\dots,-\tfrac{(W^{-1}(\zeta- \xi))_n}{\lambda_{n}}\Big\},\]
where $W\in \R^{n\times n}$ is the matrix of the linear isomorphism that sends the $i$-th standard basis vector of $\mathbb{R}^{n}$ to $w_{i}$, and $\lambda_{1}, \dots ,\lambda_{n} >0$ are from \eqref{eq:vcone}.

Consider the following continuous map
\begin{align}
\rho  \colon [0,1] \times \Log(f^{-1}(\mathbb{R}_{<0})) \to \Log(f^{-1}(\mathbb{R}_{<0})), \quad (t,\zeta) \mapsto \zeta + t \, s^*(\zeta) \, v.
\end{align}
Since $v$ is a strict separating vector of $\sigma(f)$, from  Lemma~\ref{Lemma_BehavOffvx}(ii') follows that $\rho(t,\zeta) \in \Log(f^{-1}(\mathbb{R}_{<0}))$ for all $(t,\zeta)  \in [0,1] \times \Log(f^{-1}(\mathbb{R}_{<0}))$. 
Clearly,  $\rho(0,\cdot)$ is the identity map, and by definition of $s^*$,  $\rho(1,\zeta) \in \xi + \Cone(w_{1}, \dots , w_{n})$
for all $\zeta \in \Log(f^{-1}(\mathbb{R}_{<0}))$. Furthermore, if $\zeta \in \xi + \Cone(w_{1}, \dots , w_{n})$, then $s^*(\zeta)=0$ and $\rho(t,\zeta) = \zeta$ for all $t\in [0,1]$. 

We conclude that $\rho$ is a homotopy showing that  $\xi + \Cone(w_{1}, \dots , w_{n})$ is a strong deformation retract of $\Log(f^{-1}(\mathbb{R}_{<0}))$. This implies (i).

 \smallskip 
Finally, we show statement (ii). Let $x \in f^{-1}(\{0\})$. Since $v \in \mathcal{S}^{-}(f)$ and strict,   Lemma~\ref{Lemma_BehavOffvx}(ii) gives that $f_{v,x}(t) < 0$ for all $t > 1$. 
Thus the sequence $\big( \gamma_{v,x}(1+\tfrac{1}{n})\big)_{n \in \mathbb{N}}$ belongs to $f^{-1}( \R_{<0} )$. As $\gamma_{v,x}$ is continuous and $\gamma_{v,x}(1) = x$, the sequence $\big( \gamma_{v,x}(1+\tfrac{1}{n})\big)_{n \in \mathbb{N}}$ converges to $x$. So  each $x \in f^{-1}( \R_{\leq 0})$ is the limit of a convergent sequence in $f^{-1}( \R_{<0} )$. Hence $f^{-1}( \R_{\leq 0}) \subseteq \overline{f^{-1}( \R_{<0} )}$. The other inclusion is clear by the continuity of $f$. 
\end{proof}

\begin{figure}
\centering
\begin{minipage}[h]{0.3\linewidth}

\centering
\includegraphics[scale=0.5]{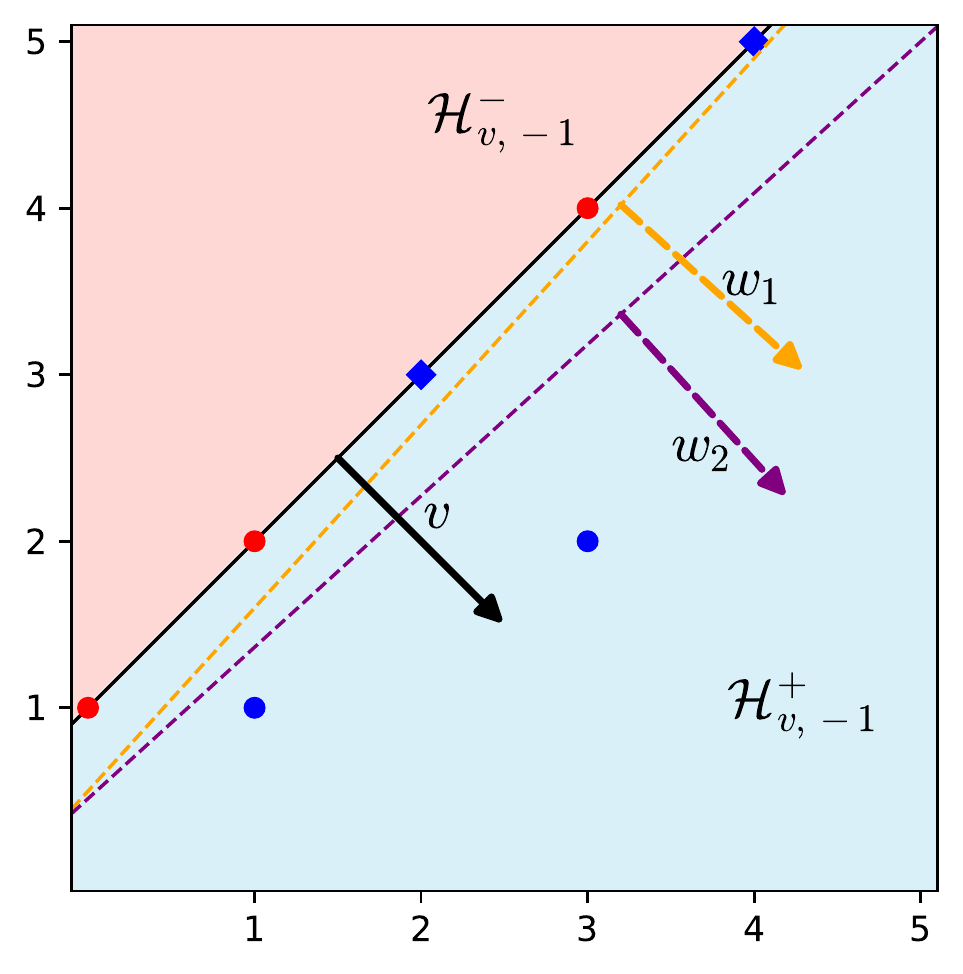}

{\small (a)}

\end{minipage}

\begin{minipage}[h]{0.45\linewidth}

\centering
\includegraphics[scale=0.5]{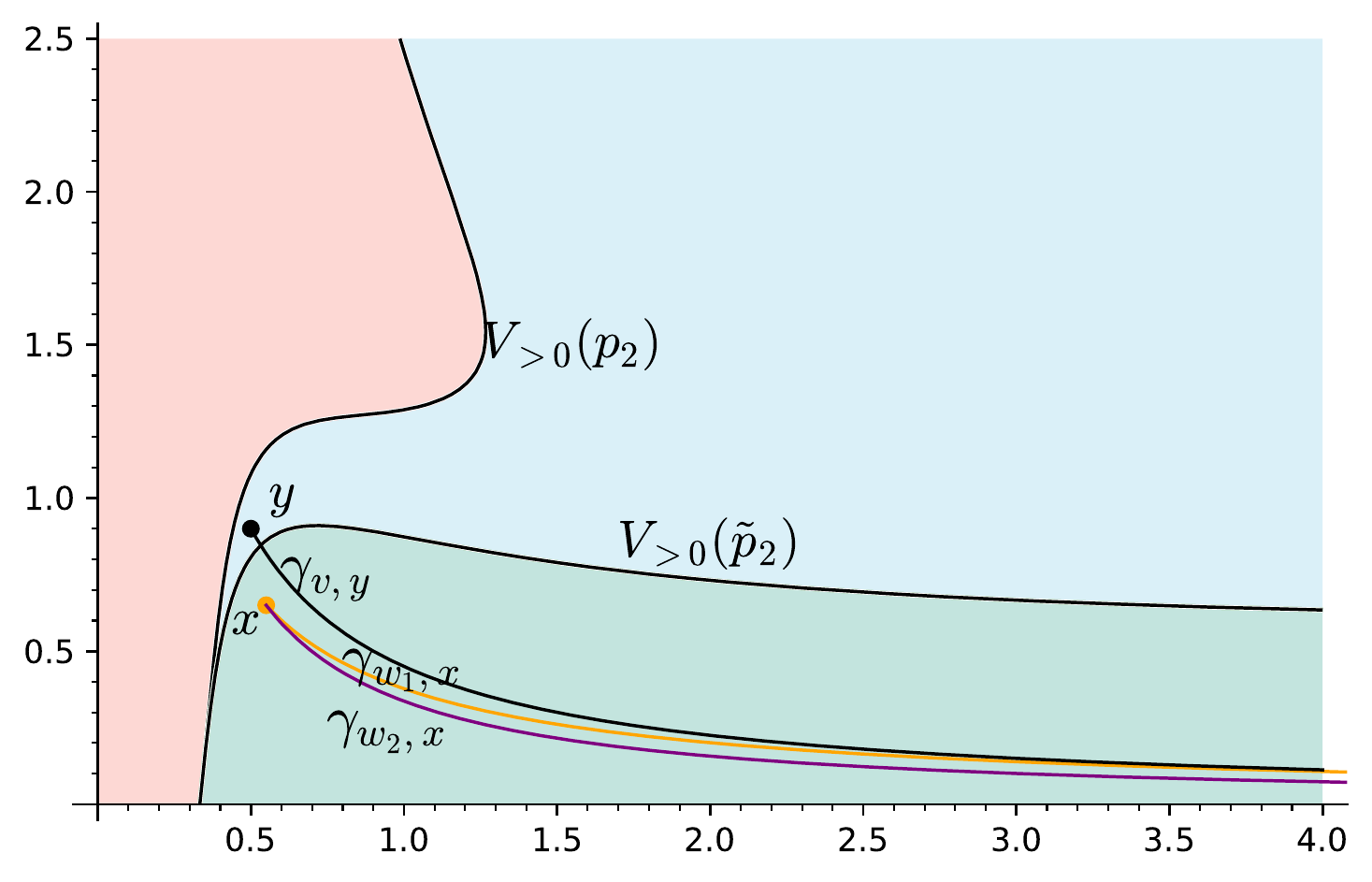}

{\small (b)}
\end{minipage}
\begin{minipage}[h]{0.45\linewidth}

\centering
\includegraphics[scale=0.5]{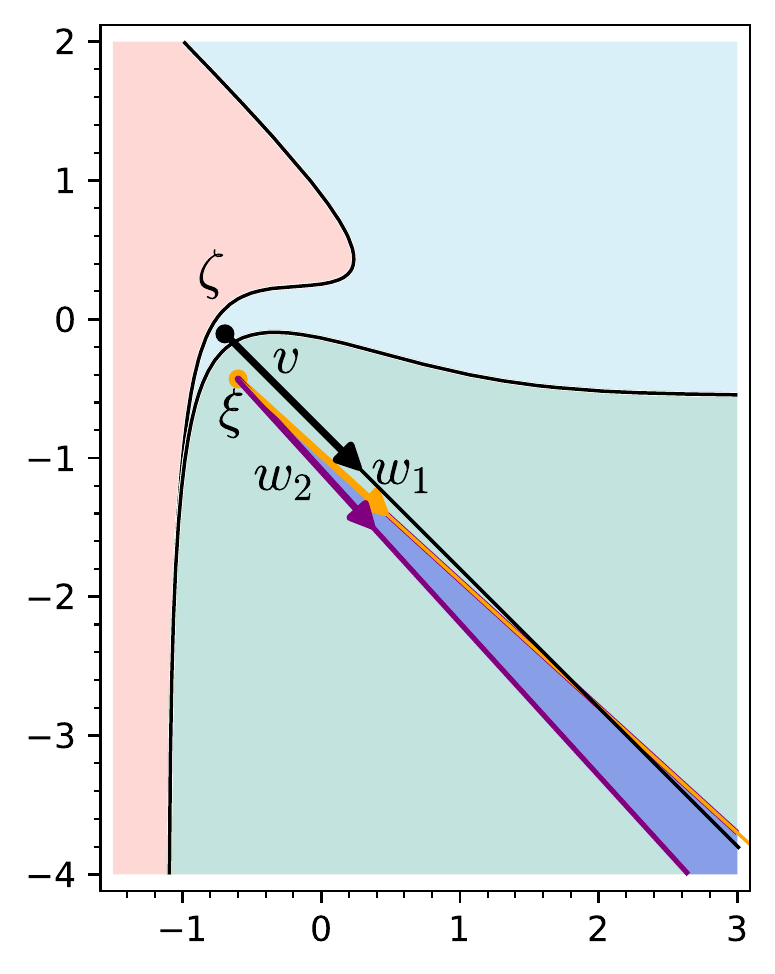}

{\small (c)}

\end{minipage}

\caption{{\small Graphical representation of Example~\ref{Ex_Thm10}.  (a) $v = (1,-1)\in \mathcal{S}^{-}(p_2)$ is a strict separating vector, the vectors $w_{1} = (1.1,-1)$ and  $w_{2} = (1,-1.1)$ are very strict separating vectors of the support of  $\tilde{p}_{2}(x_{1},x_{2})$ and form a basis of $\R^2$. (b) $p_2^{-1}(\R_{<0})$ shown in blue and its subset $\tilde{p}_2^{-1}(\R_{<0})$ shown in green. (c) The half-line $\Log(\gamma_{v,y})$  intersects the cone generated by $w_{1}, w_{2}$ with apex $\xi = \Log(x)$. 
}}  \label{FIG3}
\end{figure}

\begin{ex}
\label{Ex_Thm10}
Consider the signomial 
 \[ p_{2}(x_{1},x_{2}) = -x_{1}^{4}x_{2}^{5} + 3x_{1}^{3}x_{2}^{4} - x_{1}^{3}x_{2}^{2} - x_{1}^{2}x_{2}^{3} + x_{1}x_{2}^{2} - 3x_{1}x_{2} + x_{2}. \]
Then $v = (1,-1)\in \mathcal{S}^{-}(p_2)$ is strict, see Fig.~\ref{FIG3}(a), and  by Theorem~\ref{Thm_StrictSepConnected}, $V_{>0}^c(p_2)$ has one negative connected component which is a contractible set.

Fig.~\ref{FIG3} displays the idea of the proof of  Theorem~\ref{Thm_StrictSepConnected}. First, one considers the signomial obtained by removing the negative monomials on the separating hyperplane $\mathcal{H}_{v,-1}$ from  Fig.~\ref{FIG3}(a): 
\[ \tilde{p}_{2}(x_{1},x_{2}) = 3x_{1}^{3}x_{2}^{4} - x_{1}^{3}x_{2}^{2} + x_{1}x_{2}^{2} - 3x_{1}x_{2} + x_{2}.\]
Using Proposition \ref{prop_basis}, one can find strict separating vectors $w_{1} = (1.1,-1)$ and $w_{2}= (1,-1.1)$ of $\sigma(\tilde{p}_{2})$ such that $v\in \Cone(w_1,w_2)$. For a fixed $x \in \tilde{p}_{2}^{-1}(\mathbb{R}_{<0})$, the paths $\gamma_{w_{1},x}, \gamma_{w_{2},x}$ turn into half-lines with start point $\xi = \Log(x)$ under the coordinate-wise logarithm map (see Fig.~\ref{FIG3} (b,c)). For each point $\zeta = \Log(y) \in \Log(p_{2}^{-1}(\mathbb{R}_{<0}))$, the half-line with start point $\zeta$ and direction vector $v$ intersects $\Cone(w_1,w_2)$. By sending $\zeta$ to the first such intersection point, we obtain that $\Cone(w_1,w_2)$ is a strong deformation retract  of $\Log(p_{2}^{-1}(\mathbb{R}_{<0}))$.
\end{ex}

The results provided so far  guarantee that  $V_{>0}^{c}(f)$ has at most one negative connected component. 
With analogous techniques, the existence of strict enclosing vectors of $\sigma(-f)$ gives that $V_{>0}^{c}(f)$ has at most two negative connected components.  Note that a strict enclosing vector of $\sigma(-f)$ defines two parallel hyperplanes such that the positive points of $\sigma(f)$ are between them,  and the negative points of $\sigma(f)$ are on the other side of these hyperplanes.

\begin{thm}
\label{Thm_StripAroundPosExp}
Let $f\colon \mathbb{R}^{n}_{>0} \to \mathbb{R}$ be a signomial. If there exists a strict enclosing vector of $\sigma(-f)$, then $V_{>0}^{c}(f)$ has at most two negative connected components. 

\end{thm}

\begin{proof}
Let $v \in \mathcal{E}^{-}(-f)$ be a strict enclosing vector. Then  for $\beta\in \sigma_{+}(-f)=\sigma_{-}(f)$, it holds that either
\[  v \cdot \beta \leq \min\limits_{\alpha \in \sigma_{+}(f)} v \cdot \alpha\quad \text{  or  } \quad \max\limits_{\alpha \in \sigma_{+}(f)} v \cdot \alpha \leq v \cdot \beta. \]
As $v$ is strict, the following sets are 
non-empty:
\[ M := \{ \beta \in \sigma_{-}(f) \mid  \max\limits_{\alpha \in \sigma_{+}(f)} v \cdot \alpha < v \cdot \beta \},\qquad N := \{ \beta \in \sigma_{-}(f) \mid v \cdot \beta < \min\limits_{\alpha \in \sigma_{+}(f)} v \cdot \alpha\}.\] 
Consider the restriction of $f$ to the   sets $M\cup  \sigma_{+}(f)$ and $N\cup  \sigma_{+}(f)$: 
\begin{align*}
\tilde{f}_{M} & := f_{|M\cup  \sigma_{+}(f)}  &  \tilde{f}_{N} &:= f_{|N\cup  \sigma_{+}(f)}.  
\end{align*}
By construction, see \eqref{eq:Salg}, $v$ and $-v$ are strict separating vectors of $\sigma(\tilde{f}_{M})$ and 
$\sigma(\tilde{f}_{N})$ respectively.
Hence $\tilde{f}^{-1}_{N}( \R_{<0})$ and $\tilde{f}^{-1}_{M}( \R_{<0})$ are path connected by Theorem \ref{Thm_StrictSepConnected}. 
Additionally, as the sets of negative points in $\sigma(\tilde{f}_{M})$ and $\sigma(\tilde{f}_{N})$ are included in $\sigma_-(f)$, it holds $f(x) \leq \tilde{f}_{N}(x)$ and $f(x) \leq \tilde{f}_{M}(x)$ for all $x \in \mathbb{R}^{n}_{>0}$ and hence 
\[\tilde{f}^{-1}_{M}( \R_{<0} ) \subseteq f^{-1}( \R_{<0} ), \qquad 
  \tilde{f}^{-1}_{N}( \R_{<0})  \subseteq f^{-1}( \R_{<0} ). \]

With this in place, if we show that for every $x \in f^{-1}( \R_{<0} )$ there is a continuous path to a point in $\tilde{f}^{-1}_{M}( \R_{<0} )$  or to a point in $\tilde{f}^{-1}_{N}( \R_{<0} )$ and this path is contained in $f^{-1}( \R_{<0} )$, then   the number of connected components of $f^{-1}( \R_{<0})$ is at most $2$.

Fix $x \in f^{-1}( \R_{<0})$.  As $v$ is a strict separating vector of $\sigma(\tilde{f}_{M})$ 
and $-v$ of $\sigma(\tilde{f}_{N})$,  there exist $t_x,d_x>1$ such that $\gamma_{v,x}(t_{x}) \in \tilde{f}_{M}^{-1}(\R_{<0} )$ and $\gamma_{-v,x}(d_{x}) \in \tilde{f}_{N}^{-1}(\R_{<0})$
by Lemma~\ref{Lemma_BehavOffvx}(ii). 

By Lemma~\ref{Lemma_BehavOffvx}(i), $f_{v,x}$ has negative leading and smallest degree coefficients, and the coefficient sign sequence has at most two sign changes. 
Hence either $f_{v,x}(t)<0$ for all $t\geq 1$ or $f_{v,x}(t)<0$  for all $t\leq 1$. 
If $f_{v,x}(t)=f(\gamma_{v,x}(t))<0$ for all $t\geq 1$, then the path $\gamma_{v,x}$ connects $x$ to a point in 
$\tilde{f}_{M}^{-1}(\R_{<0} )$. If $f_{v,x}(t)<0$ for all $t\leq 1$, then 
$f_{-v,x}(t)=f_{v,x}(t^{-1})<0$ for all $t\geq 1$. Hence 
the path $\gamma_{-v,x}$ connects $x$ to a point in 
$\tilde{f}_{N}^{-1}(\R_{<0} )$.
This concludes the proof. 
\end{proof}

\begin{ex}
\label{Ex_PosStrip}
Consider the signomial 
\[ p_{3}(x_{1},x_{2}) = x_{1}^{3}x_{2}^{5} - x_{1}^{2}x_{2}^{5} + x_{1}^{4}x_{2}^{2} + x_{1}^{3}x_{2}^{3} - x_{1}^{5} - x_{1}x_{2}^{4} - x_{1}^{3}x_{2} + 3x_{1}^{2}x_{2}^{2} - x_{1}x_{2}^{3} + x_{1}x_{2}.\]
The vector $v = (1,-1)$ is a strict enclosing vector of $-p_3$, see Fig.~\ref{FIG4}(a).  Hence, the number of negative connected components of $V_{>0}^{c}(p_{3})$ is at most two by Theorem \ref{Thm_StripAroundPosExp}.

In Fig.~\ref{FIG4}(b), the idea of the proof of Theorem \ref{Thm_StripAroundPosExp} is illustrated. 
The following two signomials are considered
\begin{align*}
\tilde{p}_{3,M}(x_{1},x_{2})  & =  x_{1}^{3}x_{2}^{5} + x_{1}^{4}x_{2}^{2} + x_{1}^{3}x_{2}^{3} - x_{1}^{5} + 3x_{1}^{2}x_{2}^{2} + x_{1}x_{2}, \\ 
\tilde{p}_{3,N}(x_{1},x_{2}) &= x_{1}^{3}x_{2}^{5} - x_{1}^{2}x_{2}^{5} + x_{1}^{4}x_{2}^{2} + x_{1}^{3}x_{2}^{3} - x_{1}x_{2}^{4} + 3x_{1}^{2}x_{2}^{2}  + x_{1}x_{2}.
\end{align*}
For each of these signomials,  the pre-image of $\R_{<0}$ is path connected and  contained in $p_{3}^{-1}( \R_{<0})$. Using the paths $\gamma_{v,x}$ or $\gamma_{-v,x}$, any point $x \in p_{3}^{-1}( \R_{<0})$ is connected to one of these two connected sets.
\end{ex}

\begin{figure}[t]
\centering
\begin{minipage}[h]{0.45\textwidth}

\centering
\includegraphics[scale=0.5]{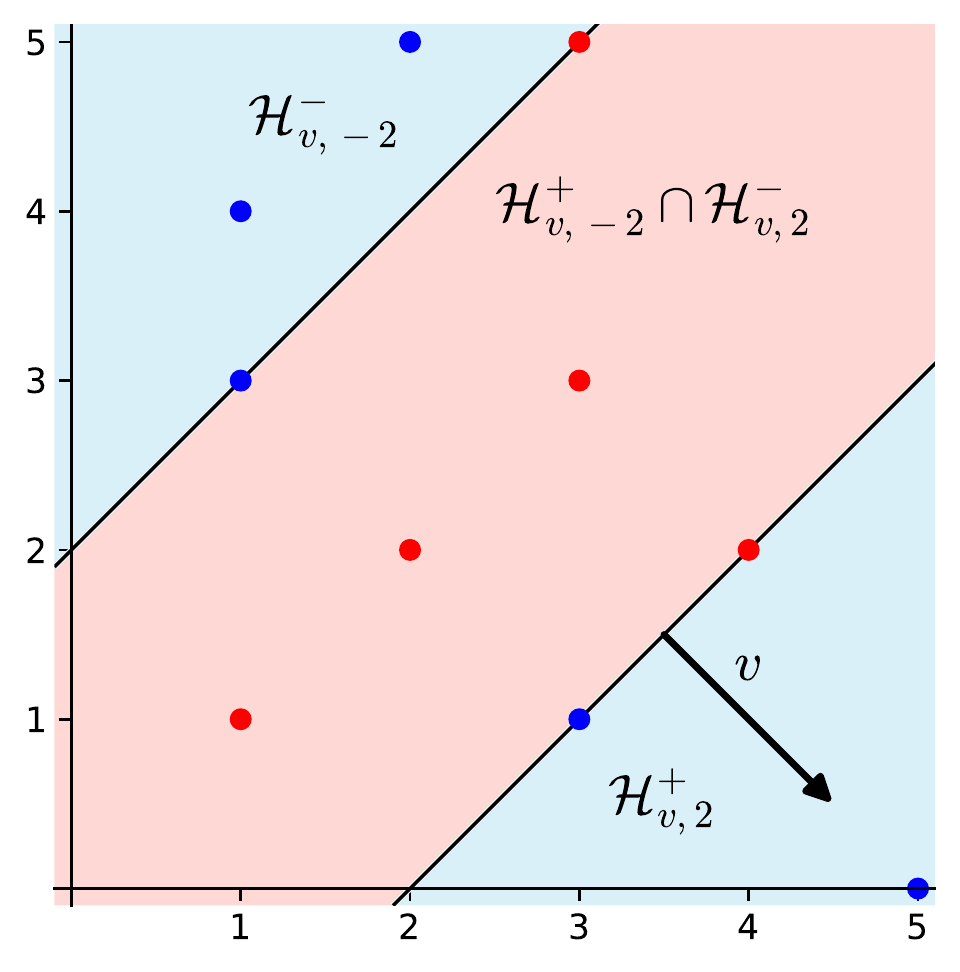}

{\small (a)}

\end{minipage}
\begin{minipage}[h]{0.45\textwidth}

\centering
\includegraphics[scale=0.5]{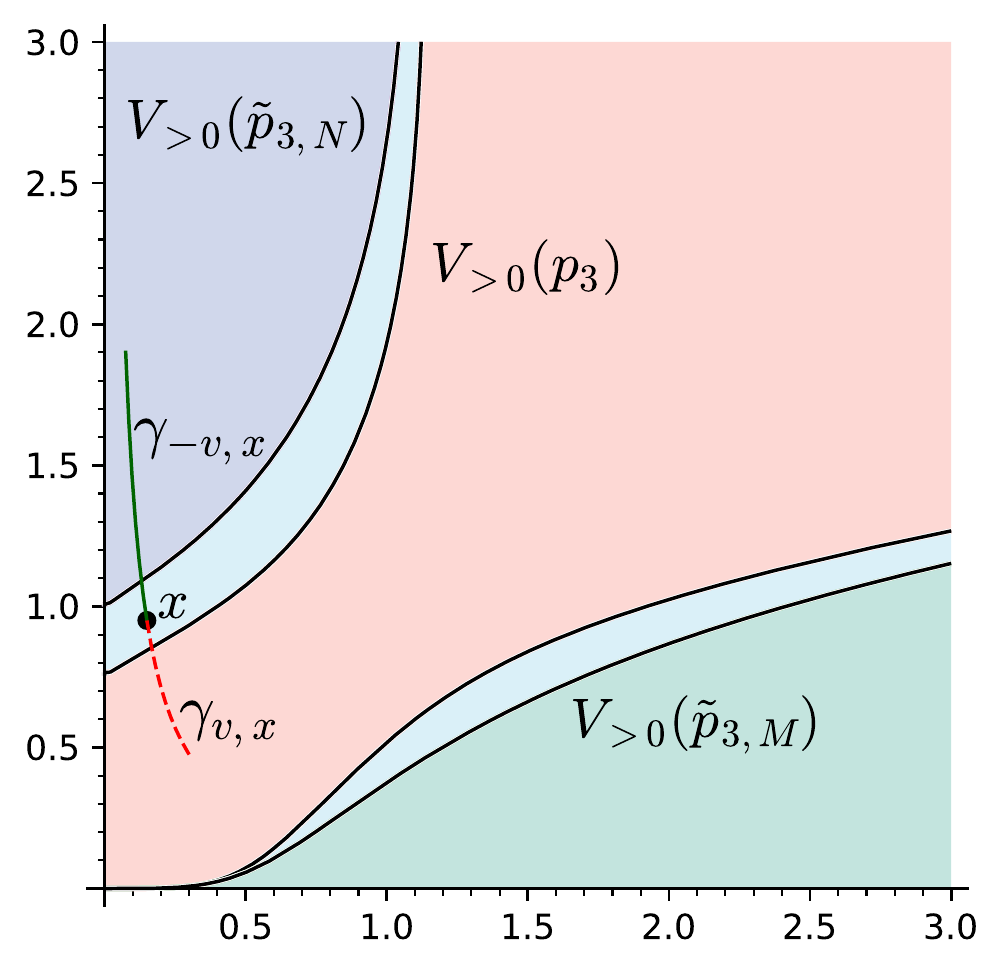}

{\small (b)}
\end{minipage}

\caption{
{\small Illustration of Example~\ref{Ex_PosStrip}.  (a) A strict enclosing vector for $-p_3$ is shown.  (b) The positive connected component of $V_{>0}^{c}(p_{3})$ is shown in red, the negative connected components of $V_{>0}^{c}(p_{3})$ are shown in blue, the subset $\tilde{p}_{3,M}^{-1}(\R_{<0})$ is shown in green, 
and the subset $\tilde{p}_{3,N}^{-1}(\R_{<0})$ is shown in purple. 
The path $\gamma_{v,x}$  from  $x = (0.15,0.95)$ to $\tilde{p}_{3,M}^{-1}( \R_{<0})$, shown dashed in red, is not contained in  $p_3^{-1}( \R_{<0})$. The path $\gamma_{-v,x}$, shown in solid green, connects $x$ with $\tilde{p}_{3,N}^{-1}( \R_{<0})$ and does not leave $p_3^{-1}( \R_{<0})$.  }} \label{FIG4}
\end{figure}

\begin{remark}
The conditions of Theorems \ref{Thm_StrictSepConnected} and \ref{Thm_StripAroundPosExp} can be checked computationally using linear programming. Finding a separating vector of $\sigma(f)$ corresponds to finding a solution of the linear inequality system
\begin{align}
\label{Eq_CondSepVector}
v \cdot \alpha \leq a, \quad \alpha \in \sigma_+(f), \qquad v \cdot \beta \geq a, \qquad \beta \in \sigma_-(f),
\end{align}
where $v \in \mathbb{R}^{n}, a \in \mathbb{R}$ are treated as unknown variables. Existing software like SageMath \cite{sagemath}, Polymake \cite{polymake:2000} and other linear programming software can find a solution to \eqref{Eq_CondSepVector} even for large number of variables and of inequalities.

Finding an enclosing hyperplane as in Theorem \ref{Thm_StripAroundPosExp} can be more demanding computationally. A naive approach is to consider all partitions of $\sigma_-(f)$ into two sets $\sigma_{-,1}(f),\sigma_{-,2}(f)$ and for each partition decide the feasibility of the system of linear inequalities
\begin{align*}
v \cdot \beta \leq a, \quad \beta \in \sigma_{-,1}(f) \qquad  a \leq     v \cdot \alpha \leq b, \quad \alpha \in \sigma_+(f), \qquad v \cdot \beta \leq b, \quad \beta \in \sigma_{-,2}(f) .
\end{align*}
\end{remark}
\color{black}

\begin{remark}
One might be tempted to believe that in the situation of Theorem~\ref{Thm_StripAroundPosExp},  $V_{>0}^c(f)$ has at most one positive connected component. However, Example~\ref{Example_2_1} gives a counter example, as  $V_{>0}^c(p_1)$ has two positive connected components, and the vector $v=(0,1)$ satisfies the hypotheses of Theorem~\ref{Thm_StripAroundPosExp}, see Fig.~\ref{FIG_Example_2_1}.
\end{remark}

A direct consequence of Theorems~ \ref{Thm_StrictSepConnected}  and \ref{Thm_StripAroundPosExp} applies to the case where the positive points of $\sigma(f)$ belong to a hyperplane that does not contain all   the negative points of $\sigma(f)$. 

\begin{cor}
\label{Cor_PosOnHyperplane} 
Let $f\colon \mathbb{R}^{n}_{>0} \to \mathbb{R}$ be a signomial. If  for some $v\in \R^n$ and $a\in \R$
\[ \sigma_+(f)\subseteq \mathcal{H}_{v,a} \qquad \textrm{and}\qquad \sigma_-(f) \nsubseteq
\mathcal{H}_{v,a},\] 
then $V_{>0}^c(f)$ has at most two negative connected components.
\end{cor}

\begin{proof}
The conditions imply that either $v$ is a strict enclosing vector of $\sigma(-f)$, or either $v$ or $-v$ is a strict separating vector of $\sigma(-f)$. The statement then follows from 
Theorem \ref{Thm_StripAroundPosExp} or 
Theorem \ref{Thm_StrictSepConnected}.
\end{proof}

\begin{cor}
\label{Cor_FewPosExp}
Let $f\colon \mathbb{R}^{n}_{>0} \to \mathbb{R}$ be a signomial. If 
\[ \# \sigma_{+}(f) \leq \dim \N(f),\] 
then $V_{>0}^c(f)$ has at most two negative connected components.
\end{cor}

\begin{proof}
Since $\# \sigma_{+}(f) \leq \dim \N(f) \leq n$, 
the points $\sigma_{+}(f)$ lie on an affine subspace of dimension at most $\dim \N(f)-1$.  Necessarily, this subspace cannot contain all points of $\sigma(f)$. 
Hence, there exists an affine hyperplane $\mathcal{H}_{v,a}$  containing $\sigma_{+}(f)$ and not containing 
$\sigma_{-}(f)$. 
Now, the statement follows from Corollary~\ref{Cor_PosOnHyperplane}.
\end{proof}

\begin{remark}
The techniques used in this section rely on the observation that the paths \eqref{Nota_Path} become half-lines at the logarithmic scale. Studying images of algebraic sets under the coordinate-wise logarithm map has a rich history. In 1994, Gelfand et al. \cite{gelfand1994discriminants} introduced the \emph{amoeba} of a Laurent polynomial $f \in \mathbb{C}[x_{1}^{\pm 1}, \dots , x_{n}^{\pm 1}]$ which is the image of the set $\{ z \in (\mathbb{C}^{*})^{n} \mid f(z) = 0 \}$ under the map $(\mathbb{C}^{*})^{n} \to \mathbb{R}^{n}, (z_{1},\dots,z_{n}) \mapsto (\log(|z_{1}|),\dots,\log( |z_{n} | ))$. Since then, many results have been proved about the structure of the connected components of the complement of the amoeba. It is known that these connected components are convex \cite[Corollary 1.6]{gelfand1994discriminants}, their number is at least equal to the number of vertices of the Newton polytope $\N(f)$ and at most equal to the total number of integer points in $\N(f) \cap \mathbb{Z}^{n}$ \cite[Theorem 2.8]{Forsberg2000LaurentDA}. Furthermore, if the polynomial is maximally sparse (i.e. every exponent of $f$ is a vertex of $\N(f)$), then the number of connected components of the complement of the amoeba is equal to the number of vertices of $\N(f)$ \cite{Nisse2008maximally}, and each of these components is unbounded \cite[Corollary 1.8]{gelfand1994discriminants}.

The logarithmic image of $V_{>0}(f)$ can be seen as the “positive real part” of the amoeba of $f$. Therefore, one might hope that statements about amoebas can be translated directly to answer Problem \ref{TheProblem}. However, logarithmic images of $V_{>0}(f)$ have been studied in \cite{Alessandrini2007LogarithmicLS}, where the author concluded that, in general, it is not possible to use properties of the amoeba to understand the logarithmic image of $V_{>0}(f)$ \cite[Section 5.1]{Alessandrini2007LogarithmicLS}. To illustrate that the amoeba of $f$ and the logarithmic image can behave differently, we recall the following example \cite[Example 2.6]{rojas2017adiscriminants}. 
Consider the maximally sparse polynomial $f = 1 - x_{1} - x_{2} + \tfrac{6}{5}  x_{1}^4 x_{2} + \tfrac{6}{5} x_{1}x_{2}^{4}$. The complement of the amoeba of $f$ has $5$ connected components, which are convex and unbounded. However, it is easy to see that the complement of $\Log(V_{>0}(f))$ has a bounded connected component, which is contained in the amoeba of $f$.
\end{remark}
\color{black}

\section{Convexification of signomials}
\label{Convexification}
In Section \ref{Section_Path_on_log_paper}, we used continuous paths (\ref{Nota_Path}), which are half-lines on logarithmic scale, to derive bounds for the number of negative connected components of $V_{>0}^{c}(f)$, where $f$ is a signomial function. In this section, we take a different approach to bound the number of negative connected components of $V_{>0}^{c}(f)$. We use the almost trivial observation that every sublevel set of a convex function is a convex set (see e.g. \cite[Theorem 4.6.]{Rockafellar}). Therefore, $V_{>0}^{c}(f)$ has at most one negative connected component, if $f$ is a convex function. With this in mind,  we investigate what signomials can be transformed into a convex function using Lemma \ref{Lemma_TransInvariant}.

From \cite[Theorem 7]{Maranas_Floudas}, one can easily derive a sufficient condition for convexity of signomials.

\begin{lemma}
\label{Lemma_ConvexCit}
A signomial $f\colon \mathbb{R}^{n}_{>0} \to \mathbb{R}$ is a convex function if the following holds:
\begin{itemize}
\item[(a)]  For each $\alpha \in \sigma_{+}(f)$, it holds that
\begin{itemize}
\item[(i)] $\alpha_{i} \leq 0$ for all $i = 1, \dots ,n$, or
\item[(ii)] there exists  $j \in \{ 1, \dots ,n \}$ such that $\alpha_{i} \leq 0$ for all $i \neq j$ and 
$(1, \dots ,1) \cdot \alpha \geq 1$,
\end{itemize}
\item[(b)] For each $\beta \in \sigma_{-}(f)$, it holds that $\beta_{i} \geq 0$ for all $i = 1, \dots ,n$ and $(1, \dots ,1)\cdot\beta \leq 1$.
\end{itemize} 
\end{lemma}

\begin{proof}
By \cite[Theorem 7]{Maranas_Floudas}, hypotheses (a) and (b) imply that each term $c_{\alpha}x^{\alpha}$, $\alpha \in \sigma_{+}(f)$ and $c_{\beta}x^{\beta}$, $\beta \in \sigma_{-}(f)$ is convex. The result follows from the fact that the sum of convex functions is convex.
\end{proof}

We proceed to interpret the conditions in Lemma \ref{Lemma_ConvexCit}   geometrically.

\begin{definition}
Given an $n$-simplex $P \subseteq \mathbb{R}^{n}$ with vertices $\mu_{0}, \dots , \mu_{n}$, we define for $k\in \{0,\dots,n\}$   the \textit{negative vertex cone} at the vertex $\mu_{k}$ as
\begin{align*}
P^{-,k} & :=  \mu_{k} + \Cone( \mu_k - \mu_0, \dots , \mu_k - \mu_n ) \\ 
&= \left\{ \sum\limits_{i = 0}^{n} \lambda_{i}\mu_{i} \mid \sum\limits_{i = 0}^{n} \lambda_{i} = 1,\  \lambda_{i} \leq 0 \text{ for all } i \neq k \right\}.
\end{align*} 
We write $P^{-} = \bigcup_{k=0}^n P^{-,k}$.
\end{definition}

Note that it follows that  $\lambda_{k} > 0$ in the definition of $P^{-,k}$. 
The name 'negative vertex cone' comes from \cite{VertexCone_Brion, VertexCone_Sottile}, where the authors refer to the \textit{vertex cone} as the pointed convex cone with apex $\mu_{k}$ and generators the edge directions pointing out of $\mu_k$.
Fig.~\ref{FIG5}(a) shows an example of the negative vertex cones in the plane. 

The next proposition  provides another geometric interpretation of negative vertex cones. First recall that every $n$-simplex $P \subseteq \mathbb{R}^{n}$ has $n+1$ facets, each facet $F$ is  supported on a hyperplane $\mathcal{H}_{v_{F},a_{F}}$, and it holds that $P = \bigcap_{F\subseteq P \text{ facet}} \mathcal{H}^{-}_{v_{F},a_{F}}$ \cite[Section 4.1]{grunbaum2003convex}.

\begin{prop}
\label{Prop_GeomInt_Pout}
Let $P = \Conv(\mu_{0}, \dots ,\mu_{n})  \subseteq \mathbb{R}^{n}$ be an $n$-simplex. A point $\alpha \in \mathbb{R}^{n}$ belongs to $P^{-,k}$ for $k\in \{0,\dots,n\}$, if and only if 
$\alpha \in \mathcal{H}^{+}_{v_{F},a_{F}}$ for all facets $F$ of $P$ containing $\mu_k$. In that case, it holds  $\alpha \in \mathcal{H}^{-}_{v_{F},a_{F}}$ for the facet $F$ not containing $\mu_k$.
\end{prop}

\begin{proof}
Denote by $F_{i}$ the facet of $P$ that does not contain $\mu_{i}$ and 
$\mathcal{H}_{v_{i},a_{i}}$ a supporting hyperplane. 
In particular  it holds that
\begin{equation}\label{eq:vc}
v_{j}\cdot\mu_{i} = a_{j} \quad \textrm{for }\quad i\neq j \qquad\textrm{and}\qquad v_{i}\cdot\mu_{i} < a_{i}, \qquad \textrm{for } \quad i=0,\dots,n.
\end{equation}

The condition in the statement is equivalent to  the existence of $k \in \{ 0 , \dots , n \}$ such that 
\begin{align}
\label{Align_Cond}
 \quad v_{i} \cdot \alpha \geq a_{i} \quad \text{    for    } \quad i \neq k.
\end{align}
Write $\alpha = \sum_{j=0}^{n} \lambda_{j}\mu_{j}$ for 
$\lambda_{0}, \dots, \lambda_{n} \in \mathbb{R}$ such that   $\sum_{j=0}^{n} \lambda_{j} = 1$. Then
\begin{equation}\label{eq:ci}
\begin{aligned}
v_{i} \cdot \alpha & = \sum_{j=0}^{n} \lambda_{j}  (v_i \cdot \mu_j) = \lambda_i (v_i \cdot \mu_i) + \sum_{j=0, j\neq i}^{n} \lambda_{j}  a_i  \\ & =  \lambda_i (v_i \cdot \mu_i) + (1-\lambda_i) a_i= a_i + \lambda_i(v_i \cdot \mu_i - a_i).
\end{aligned}
  \end{equation}
Using this, condition  \eqref{Align_Cond} holds for some $k$ if and only if 
\[  
\lambda_i(v_i \cdot \mu_i - a_i)\geq 0 \quad \text{    for    } \quad i \neq k.\]
By \eqref{eq:vc}, this holds if and only if 
$\lambda_{i} \leq 0$ for $i\neq k$, that is, if and only if  $\alpha \in P^{-,k} \subseteq P^{-}$.   As then, $\lambda_k\geq 0$, \eqref{eq:ci} gives that $v_k\cdot \alpha<a_k$ and hence $\alpha\in \mathcal{H}^{-}_{v_{k},a_{k}}$. 
\end{proof}

 We write $\Delta_{n} := \Conv(e_{0},e_{1}, \dots ,e_{n})$ for the standard $n$-simplex in $\mathbb{R}^{n}$, where $e_{1}, \dots , e_{n}$ are the standard basis vectors of $\mathbb{R}^{n}$ and $e_{0}$ denotes the zero vector.

\begin{lemma}\label{Lemma_ConvCondOut}
Let $f\colon \mathbb{R}^{n}_{>0} \to \mathbb{R}$ be a signomial. 
If $\sigma_{-}(f) \subseteq \Delta_{n}$ and $\sigma_{+}(f) \subseteq \Delta_{n}^{-}$, then  $f$ is a convex function.
\end{lemma}

\begin{proof}
We show that the conditions in Lemma~\ref{Lemma_ConvexCit} are equivalent to $\sigma_{-}(f) \subseteq \Delta_{n}$ and $\sigma_{+}(f) \subseteq \Delta_{n}^{-}$.
For $\beta \in \R^n$, find  the unique $\lambda_{0}, \dots , \lambda_{n} \in \mathbb{R}$ such that $\sum_{i = 0}^{n}\lambda_{i}e_{i} = \beta$ and $\sum_{i = 0}^{n}\lambda_{i}=1$. 
Note that $(1, \dots , 1) \cdot \beta = \sum_{i = 1}^{n}\lambda_{i} =1-\lambda_0$, which is at most $1$ if and only if $\lambda_0\geq 0$.

Lemma \ref{Lemma_ConvexCit}(b)  holds if and only if 
$\lambda_{i}\geq 0$ for all $i=1,\dots,n$ and $\sum_{i = 1}^{n}\lambda_{i} \leq 1$. Equivalently, 
$\lambda_{i}\geq 0$ for all $i=1,\dots,n$ and $\lambda_0\geq 0$, that is, $\beta\in \Delta_n$. 

We show now that $\beta \in \Delta^{-}_n$ if and only if Lemma~\ref{Lemma_ConvexCit}(a)  holds. 
By definition, $\beta \in \Delta^{-}_n$ if and only if   for some $k$,
\begin{equation}\label{eq:cond1}
 \lambda_{i} \leq 0\quad  \textrm{ for }\quad i \neq k.
\end{equation}

For $k= 0$, \eqref{eq:cond1} holds if and only if $\beta_i\leq 0$ for all $i$, thus  Lemma \ref{Lemma_ConvexCit}(a,i) holds. For $k > 0$,  
  \eqref{eq:cond1} holds, if and only if all but the $k$-th coordinate of $\beta$ are non-positive, and 
$\lambda_0\leq 0$, equivalently $(1, \dots , 1) \cdot \beta \geq  1$, which is Lemma \ref{Lemma_ConvexCit}(a,ii). 
 This concludes the proof.
 \end{proof}

We next look into what signomials can be transformed into a convex signomial using the transformations from Lemma \ref{Lemma_TransInvariant}.  It is well known that any two $n$-simplices are affinely isomorphic \cite{Ziegler_book}. The next lemma shows that the negative vertex cones are preserved under such an affine transformation. 
 
\begin{lemma}
\label{Lemma_PoutQout}
Let $P,Q \subseteq \mathbb{R}^{n}$ be $n$-simplices. For every  $B \subseteq P$ and $A \subseteq P^{-}$, there exist an invertible matrix $M \in \GL_{n}(\mathbb{R})$ and a vector $v \in \mathbb{R}^{n}$ such that $M B + v \subseteq Q$ and $MA+v \subseteq Q^{-}$.
\end{lemma}

\begin{proof}
Denote by $\{p_{0}, \dots , p_{n} \}$ and $\{q_{0}, \dots , q_{n} \}$ the vertex sets of $P$ and $Q$ respectively. Since $P$ and $Q$ are simplices, there is an invertible matrix $M \in \GL_{n}(\mathbb{R})$ such that $M(p_{i} - p_{0}) = q_{i} - q_{0}$ for $i = 1, \dots , n$. Define $v := -Mp_{0} + q_{0}$. By construction, it holds that $Mp_{i} + v = q_{i}$ for every $i = 0, \dots, n$.

For each $\mu \in \mathbb{R}^{n}$, write $\mu = \sum_{i=0}^{n}\lambda_{i}p_{i}$ with $\sum_{i=0}^{n}\lambda_{i} = 1$. It holds that 
\[ M\mu + v = \sum_{i=0}^{n}\lambda_{i}Mp_{i} + \sum\limits_{i=0}^{n}\lambda_{i} v = \sum_{i=0}^{n}\lambda_{i}(Mp_{i} +v) = \sum_{i=0}^{n}\lambda_{i}q_{i}.\]
That is, the coordinates of $\mu$ according to $P$ and those of $M\mu + v$ according to $Q$ are the same. 
From this the statement follows.
\end{proof}

\begin{thm}
\label{Thm_Convexification}
Let $f\colon \mathbb{R}^{n}_{>0} \to \mathbb{R}$ be a signomial. If there exists an $n$-simplex $P$ such that 
\[ \sigma_{-}(f) \subseteq P,\qquad \textrm{and}\qquad  \sigma_{+}(f) \subseteq P^{-},\]
 then $f^{-1}( \R_{<0})$ is  either empty or contractible. In particular, $V_{>0}^{c}(f)$ has at most one negative connected component.
\end{thm}

\begin{proof}
By Lemma \ref{Lemma_PoutQout} with $B=\sigma_{-}(f)$ and $A=\sigma_{+}(f)$, there exists  $M \in \GL_{n}(\mathbb{R})$ and   $v \in \mathbb{R}^{n}$ such that $M \sigma_{-}(f) + v \subseteq \Delta_{n}$ and $M\sigma_{+}(f)+v \subseteq \Delta_{n}^{-}$.
By Lemma~\ref{Lemma_TransInvariant}, $\sigma_{+}(F_{M,v,f}) = M \sigma_{+}(f)+v$ and $\sigma_{-}(F_{M,v,f}) = M \sigma_{-}(f)+v$. Hence by Lemma~\ref{Lemma_ConvCondOut}, $F_{M,v,f}$ is a convex function and thus $F_{M,v,f}^{-1}(\R_{<0})$ is  either empty or contractible. By Lemma~\ref{Lemma_TransInvariant}
again, $f^{-1}(\R_{<0})$ is homeomorphic to $F_{M,v,f}^{-1}(\R_{<0})$, and the statement of the theorem follows.
\end{proof}

In view of Theorem~\ref{Thm_Convexification}, understanding $P^{-}$ for a simplex $P$ allows us to determine whether $f$ can be transformed to a convex function. 

\begin{ex}\label{Ex_Simplex}
Consider the signomial 
 \[ p_{4}(x_{1},x_{2}) = x_{1}^{5}x_{2}^{2} + x_{1}x_{2}^{5} - 2x_{1}^{3}x_{2}^{2} - 3x_{1}^{2}x_{2}^{2} + x_{1}x_{2}^{3} + x_{2}^{4} - x_{1}x_{2} + 1\]
and the simplex $P = \Conv( (1,1),(4,2),(1,3))$. We have $\sigma_-(p_4) \subseteq P$ and $\sigma_+(p_4) \subseteq P^{-}$, see Fig.~\ref{FIG5}. 
 By Theorem \ref{Thm_Convexification}, the set $p_{4}^{-1}( \R_{<0})$ is contractible, since $p_4(1,1) = -1$.
\end{ex}

\begin{figure}[t]
\centering
\begin{minipage}[h]{0.45\textwidth}

\centering
\includegraphics[scale=0.55]{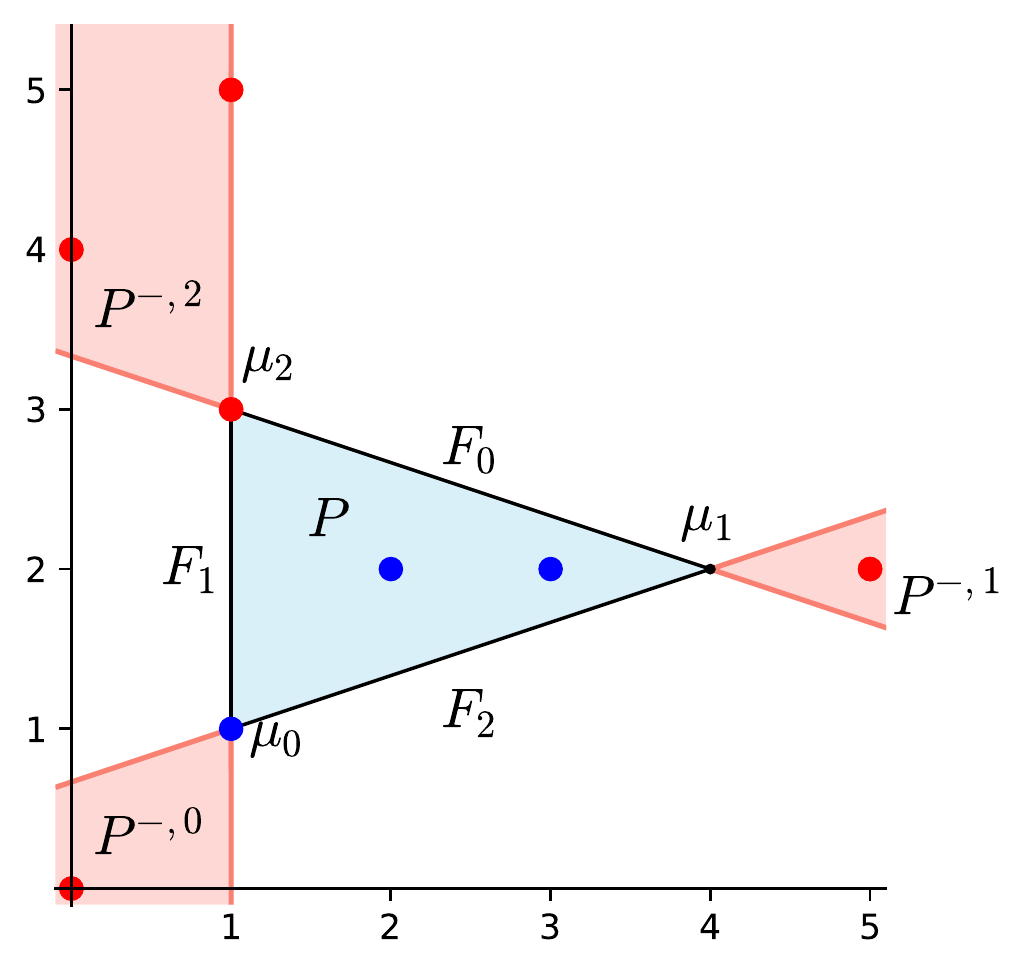}

{\small (a)}

\end{minipage}
\begin{minipage}[h]{0.45\textwidth}

\centering
\includegraphics[scale=0.55]{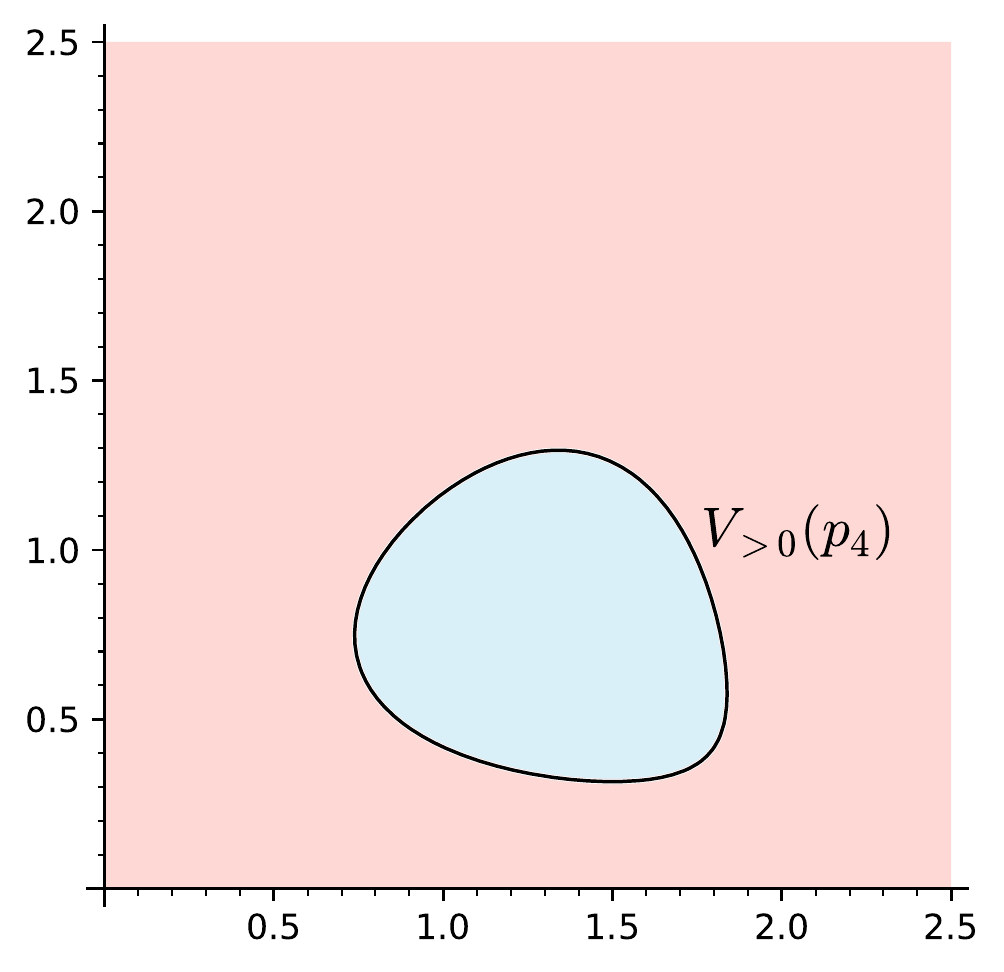}

{\small (b)}
\end{minipage}
\caption{{\small Illustration of Example~\ref{Ex_Simplex}. (a) A $2$-simplex $P$, its negative cones $P^{-}$ and the support  of $p_{4}(x_{1},x_{2})$. (b) The set $p_4^{-1}( \R_{<0})$ is shown in blue.} }\label{FIG5}
\end{figure} 

A direct consequence of Theorem~\ref{Thm_Convexification} states that if all positive points of $\sigma(f)$ are vertices of the Newton polytope and this is a simplex, then $f^{-1}( \R_{<0})$ is  either empty or contractible.
Let  $\Vertex(\N(f))$ denote the set of vertices of $\N(f)$.

\begin{cor}
\label{Cor_SimplexNewtonPoly}
Let $f\colon \mathbb{R}^{n}_{>0} \to \mathbb{R}$
be a signomial. If  $\sigma_{+}(f) \subseteq \Vertex( \N(f))$ and 
$\N(f)$ is a simplex, then $f^{-1}( \R_{<0})$ is  either empty or contractible.
\end{cor}

\begin{proof}
Let $d := \dim \N(f)$ and denote by $e_{1}, \dots ,e_{d}$ the first $d$ standard basis vectors of  $\mathbb{R}^{n}$. Without loss of generality, we can assume that $\sigma(f)$ belongs to the linear subspace generated by $e_{1}, \dots ,e_{d}$ in $\mathbb{R}^{n}$, as this can be achieved via a change of variables as in Lemma \ref{Lemma_TransInvariant}. Hence $f$ depends only on the variables $x_1,\dots,x_d$, and can be seen as a signomial in $\R^d_{>0}$ with full dimensional Newton polytope. 
Viewing $V^c_{>0}(f)$ in $\R^d_{>0}$, the statement   follows from Theorem~\ref{Thm_Convexification}, since $\sigma_{+}(f) \subseteq \Vertex( \N(f) ) \subseteq \N(f)^{-}$ and $\sigma_{-}(f) \subseteq \N(f)$.
 
 The proof is  completed noticing that the pre-image of a  contractible subset of $\mathbb{R}^{d}_{>0}$ under the projection map
$(x_{1}, \dots ,x_{n}) \mapsto (x_{1}, \dots , x_{d})$ is contractible in $\mathbb{R}^{n}_{>0}$.
\end{proof}

\begin{remark}
Finding a simplex $P$ that satisfies the conditions of Theorem \ref{Thm_Convexification} might be challenging even in low dimensions. For a partition of $\sigma_+(f)$ into $n+1$ sets $\sigma_{+,0}(f), \dots, \sigma_{+,n}(f)$, Proposition \ref{Prop_GeomInt_Pout} give rise to  a system of linear inequalities that the normal vectors of the facets of $P$ need to satisfy to ensure that $\sigma_{-}(f) \subseteq P$ and $\sigma_{+,i}(f) \subseteq P^{-,i}$ for $i = 0, \dots ,n$. To verify that a solution of this system gives indeed an $n$-simplex, one can employ Lemma~\ref{Lemma_HyperplaneSimplex} below, whose proof is given for completeness.

Using these observations, the existence of a simplex $P$ satisfying the conditions of Theorem \ref{Thm_Convexification} can be established by verifying the feasibility of a system of polynomial inequalities. This can be for example achieved using quantifier elimination \cite{quan_eli}; see \cite{redlog} for an implementation.
\end{remark}

\begin{lemma}
\label{Lemma_HyperplaneSimplex}
Let $\{\mathcal{H}_{w_{0},a_{0}}, \dots, \mathcal{H}_{w_{n},a_{n}}\}$ be a set of hyperplanes of $\mathbb{R}^{n}$ such that:
\begin{itemize}
\item[$(i)$] Every proper subset of  $\{ w_{0}, \dots ,w_{n} \}$ is linearly independent.
\item[$(ii)$] For every $i \in \{0, \dots , n \}$ it holds that $\bigcap_{j=0,j\neq i}^{n} \mathcal{H}_{w_{j},a_{j}} \subseteq \mathcal{H}^{-,\circ}_{w_{i},a_{i}}$.
\end{itemize}
Then $\bigcap_{j=0}^{n} \mathcal{H}^{-}_{w_{j},a_{j}}$ is an $n$-simplex.
 \end{lemma}
\begin{proof}
First, note that (ii) implies
\[ (ii') \qquad \bigcap_{j=0}^{n} \mathcal{H}_{w_{j},a_{j}} = \emptyset.\]
As a finite intersection of closed half-spaces,  $P := \bigcap_{j=0}^{n} \mathcal{H}^{-}_{w_{j},a_{j}}$ is a convex polyhedron. Each face of $P$ has the form 
\[ P_{I} = P \cap H_{I}, \qquad H_{I} = \bigcap_{i\in I}   \mathcal{H}_{w_{i},a_{i}},\] 
for some non-empty  subset $I \subseteq \{0, \dots ,n\}$.  By (i) and (ii'), $H_{I}$ is zero dimensional if and only if 
 $I$ has $n$ elements.
By (ii), for $ I =\{0, \dots ,n\} \setminus \{i\}$, $P_I \neq \emptyset$ and hence $P_I$ is a vertex of $P$, denoted by $\mu_i$. 
Furthermore, the points $\mu_0,\dots,\mu_n$ are affinely independent. This follows from (ii'), as for each $k$, $\mu_{i}\in \mathcal{H}_{w_{k},a_{k}}$ for $i\neq k$ and $\mu_{k}\notin \mathcal{H}_{w_{k},a_{k}}$.
Hence $\Conv(\mu_{0},\dots, \mu_{n})$ is  an $n$-simplex.
Finally, $P=\Conv(\mu_{0},\dots, \mu_{n})$ as $\mathcal{H}_{w_{k},a_{k}}^{-,\circ}$ contains a vertex for each $k$. 
\end{proof}

We conclude the section with Proposition~\ref{Lemma_n-1nonstrict}, which states that if there are $n-1$ linearly independent non-strict separating vectors and the convex hull of the negative points does not contain positive points, then a simplex satisfying the conditions of Theorem~\ref{Thm_Convexification} exists. This case, together with the scenario with one negative point in Theorem~\ref{Thm_OneExponent} or the existence of a strict separating vector in Theorem~\ref{Thm_StrictSepConnected}, conform the situations where one can effectively conclude that $V_{>0}^c(f)$ has at most one negative connected component. 
\color{black}

\begin{prop}
\label{Lemma_n-1nonstrict}
Let $f\colon \mathbb{R}^{n}_{>0} \to \mathbb{R}$ be a signomial,
such that $\sigma(f)$ has at least two negative points. 
Assume that there exist $n-1$ linearly independent  separating vectors of $\sigma(f)$, which are not strict and that $\Conv(\sigma_{-}(f)) \cap \sigma_{+}(f) = \emptyset$.  
Then there exists an $n$-simplex $P$ such that $\sigma_{-}(f) \subseteq P$ and $\sigma_{+}(f) \subseteq P^{-}$.
\end{prop}

\begin{proof} 
Let $w_{1},\dots,w_{n-1}$ be non-strict separating vectors. 
Then with   $a_{i} := \max\{w_{i}\cdot \alpha \mid \alpha \in \sigma_{+}(f)\}$, 
it holds 
\begin{align}
\label{Align_IntersWi}
\sigma_{+}(f) \subseteq  \bigcap\limits_{i = 1}^{n-1} \mathcal{H}^{-}_{w_{i},a_{i}} \quad \text{and}  \quad \sigma_{-}(f) \subseteq  L \qquad \textrm{with} \quad L:=\bigcap\limits_{i = 1}^{n-1} \mathcal{H}_{w_{i},a_{i}}.
\end{align}

If $\sigma_{-}(f) \subseteq  L$, then any simplex $P$ having as edge $\Conv(\sigma_{-}(f))$ satisfies the statement. 
Hence, we assume that this is not the case. We prove the proposition by applying Lemma \ref{Lemma_HyperplaneSimplex}. We introduce the following:
\begin{align*}
v &:= \sum_{i=1}^{n-1}w_{i} \in \R^{n-1},&   d &:= \sum_{i=1}^{n-1}a_{i}\in \R , & 
K &:= \max\, \{ v\cdot \alpha \mid \alpha \in \sigma_{+}(f),\ v\cdot\alpha \neq d\} \in \R.
\end{align*}
By assumption, $\epsilon := d - K >0$ and we have $\sigma_-(f)\subseteq \mathcal{H}_{v,d}$. 
 Let $z \in \mathbb{R}^{n}$   such that $z,w_{1},\dots,w_{n-1}$ are linearly independent, and denote by $\beta_{0}$,  $\beta_1$ the vertices of $\Conv( \sigma_{-}(f))$ where the linear form induced by $z$ attains its minimum  and its maximum respectively.  These vertices are different, otherwise each $\beta \in \Conv(\sigma_{-}(f))$ would be the  unique solution of $z \cdot \beta  = z \cdot \beta_{0} $, $w_{i} \cdot \beta = a_{i}$, $i=1,\dots ,n-1$. This would be a contradiction, since $\sigma_{-}(f)$ contains at least two points.

We let $M := \max\{ z \cdot \alpha \mid \alpha \in \sigma_{+}(f)\}$, choose $\lambda > \mu$ positive real numbers such that 
\begin{align} \label{eq:lambdamu}
\lambda(M-z\cdot \beta_{0}) \leq \epsilon &= d - K,  & 
\mu(M-z\cdot \beta_1) \leq \epsilon &= d - K,
\end{align}
and define $w_{0} := v+\lambda z$,  $w_{n} := -v-\mu z$,  $a_{0} := d + \lambda (z\cdot \beta_{0})$,  and $a_{n} := -d - \mu (z\cdot \beta_1)$. By construction, $\beta_0\in \mathcal{H}_{-w_{0},-a_{0}}$ and $\beta_1\in \mathcal{H}_{-w_{n},-a_{n}}$.

We show that $P := \bigcap_{i=0}^{n} \mathcal{H}^{-}_{-w_{i},-a_{i}}$ is an $n$-simplex using Lemma \ref{Lemma_HyperplaneSimplex}, and  satisfies the hypotheses of the statement. Lemma \ref{Lemma_HyperplaneSimplex}(i) holds by construction.
To show Lemma \ref{Lemma_HyperplaneSimplex}(ii), we consider first $i \in \{0,n\}$. As 
\begin{equation}\label{eq:0n}
\bigcap_{j=0}^{n-1} \mathcal{H}_{-w_{j},-a_{j}} = \{ \beta_{0} \}, \qquad \bigcap_{j=1}^{n} \mathcal{H}_{-w_{j},-a_{j}} = \{ \beta_1 \},
\end{equation} it suffices to show that  $\beta_{0} \in \mathcal{H}^{-,\circ}_{-w_{n},-a_{n}}$  and $\beta_1 \in \mathcal{H}^{-,\circ}_{-w_{0},-a_{0}}$. For each $ \beta \in \sigma_{-}(f)$, it holds that 
\begin{align}
w_{n} \cdot \beta & = -v\cdot \beta - \mu (z \cdot \beta) \geq -d -\mu (z \cdot \beta_1 )= a_{n}, \quad \text{and} \label{Align_WnBetaCn}\\
w_{0} \cdot \beta & = v\cdot \beta + \lambda (z \cdot \beta) \geq d + \lambda (z \cdot \beta_{0}) = a_{0}  \label{Align_WnBetaC0}
\end{align}
as  $z$ attains its minimum resp. its maximum on $\Conv( \sigma_{-}(f))$ at $\beta_{0}$ resp. at $\beta_1$ and $\lambda, \mu > 0$. From these we get that $\beta_{0} \in \mathcal{H}^{-,\circ}_{-w_{n},-a_{n}}$  and $\beta_1 \in \mathcal{H}^{-,\circ}_{-w_{0},-a_{0}}$, since $z \cdot \beta_{1} > z \cdot \beta_{0}$ and hence the inequalities in \eqref{Align_WnBetaCn} and \eqref{Align_WnBetaC0} are strict.   

Consider now $i \in \{ 1, \dots , n-1 \}$ and $x \in \bigcap_{j=0,j\neq i}^{n} \mathcal{H}_{-w_{i},-a_{i}}$. 
In particular, $x \in \mathcal{H}_{w_{0},a_{0}} \cap \mathcal{H}_{w_{n},a_{n}}$. 
Solving the linear system $w_{0} \cdot x = v \cdot x + \lambda (z \cdot x) = a_{0}$ and $w_{n} \cdot x = -v \cdot x - \mu (z \cdot x) = a_{n}$ for $v \cdot x$ and $z \cdot x$ and using the definition of $a_0,a_n$, we obtain
\[ z \cdot x = \tfrac{a_{0} + a_{n}}{\lambda - \mu},\qquad v \cdot x = a_{0} -  \lambda \cdot \tfrac{a_{0} + a_{n}}{\lambda - \mu} = 
d +  \tfrac{\lambda \mu}{\lambda - \mu} ( z \cdot \beta_1 - z \cdot \beta_{0} ) > d,\] 
as $\lambda,\mu, \lambda - \mu, z \cdot \beta_1 - z \cdot \beta_{0} > 0$.
Hence 
\[ \sum\limits_{j=1}^{n-1} w_{j} \cdot x = v  \cdot x > d = \sum\limits_{j=1}^{n-1} a_{j}.\]
From this follows that $w_{i} \cdot x > a_{i}$, since $w_{j} \cdot x = a_{j}$ for $j \neq i$. Therefore $x \in \mathcal{H}^{-,\circ}_{-w_{i},-a_{i}}$ and Lemma \ref{Lemma_HyperplaneSimplex}(ii) holds. 
We conclude that  $P$ is an $n$-simplex. 
 
Finally, we show that  $\sigma_{-}(f) \subseteq P$ and  $\sigma_{+}(f) \subseteq P^{-}$. The inclusion $\sigma_{-}(f) \subseteq P$ follows  from \eqref{Align_IntersWi}, \eqref{Align_WnBetaCn} and \eqref{Align_WnBetaC0}.
 
Let $\alpha \in \sigma_{+}(f)$ and  assume that $v\cdot \alpha < d$. By \eqref{eq:lambdamu},
\[ w_{0} \cdot \alpha = v\cdot \alpha + \lambda (z \cdot \alpha) \leq K + \lambda M = d- \epsilon +\lambda M \leq d + \lambda (z\cdot \beta_{0}) = a_{0},\] 
which implies $\alpha \in \mathcal{H}^{+}_{-w_{0},-a_{0}}$. 
This together with (\ref{Align_IntersWi}) imply that  $\alpha \in P^{-}$ by  Proposition \ref{Prop_GeomInt_Pout}.
 
Now, consider the case $v\cdot \alpha = d$.  In this case,  (\ref{Align_IntersWi}) implies that $w_{i}\cdot \alpha = a_{i}$ for  each $i=1,\dots,n-1$. Thus, $\alpha \in L $ and recall $\alpha\notin \Conv(\sigma_{-}(f) )$. Hence $\alpha\in L\setminus  \Conv(\sigma_{-}(f)) \subseteq P^{-}$, where the last inclusion follows from the fact that the supporting hyperplanes of each cone  $P^{-,k}$ are supporting hyperplanes of $P$. 
\end{proof}

\begin{figure}[t]
\begin{center}

\begin{minipage}[b]{0.5\linewidth}
\centering
\includegraphics[scale=0.5]{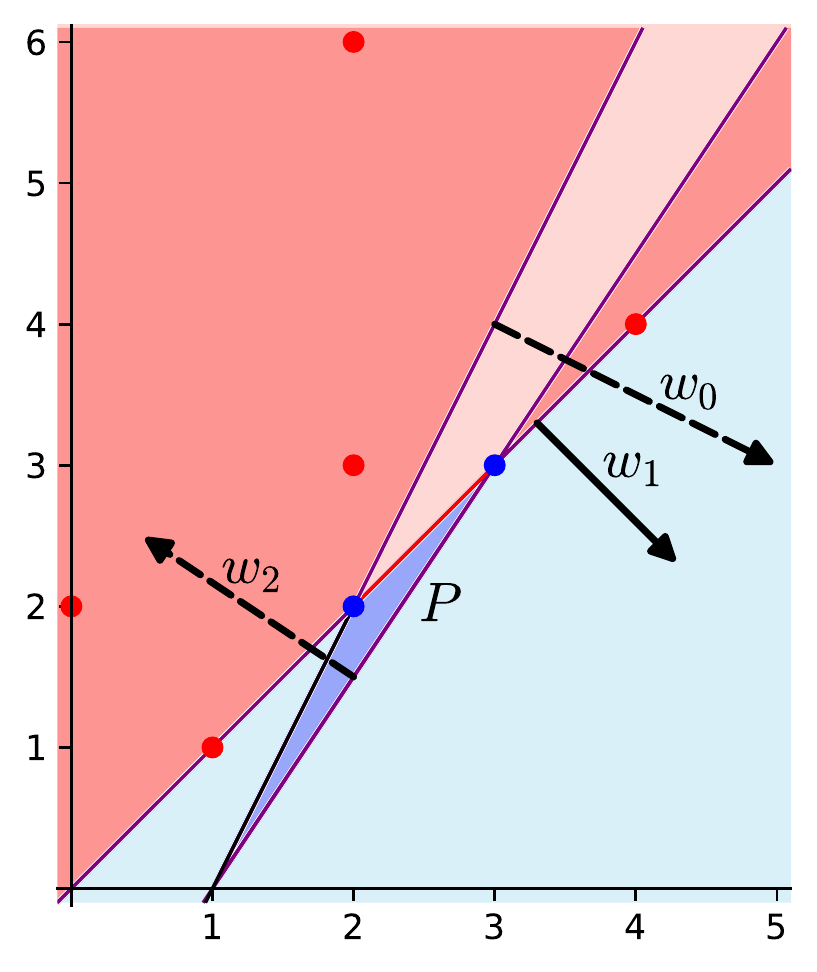}

{\small (a)}
\end{minipage}
\begin{minipage}[b]{0.3\linewidth}
\centering
\includegraphics[scale=0.4]{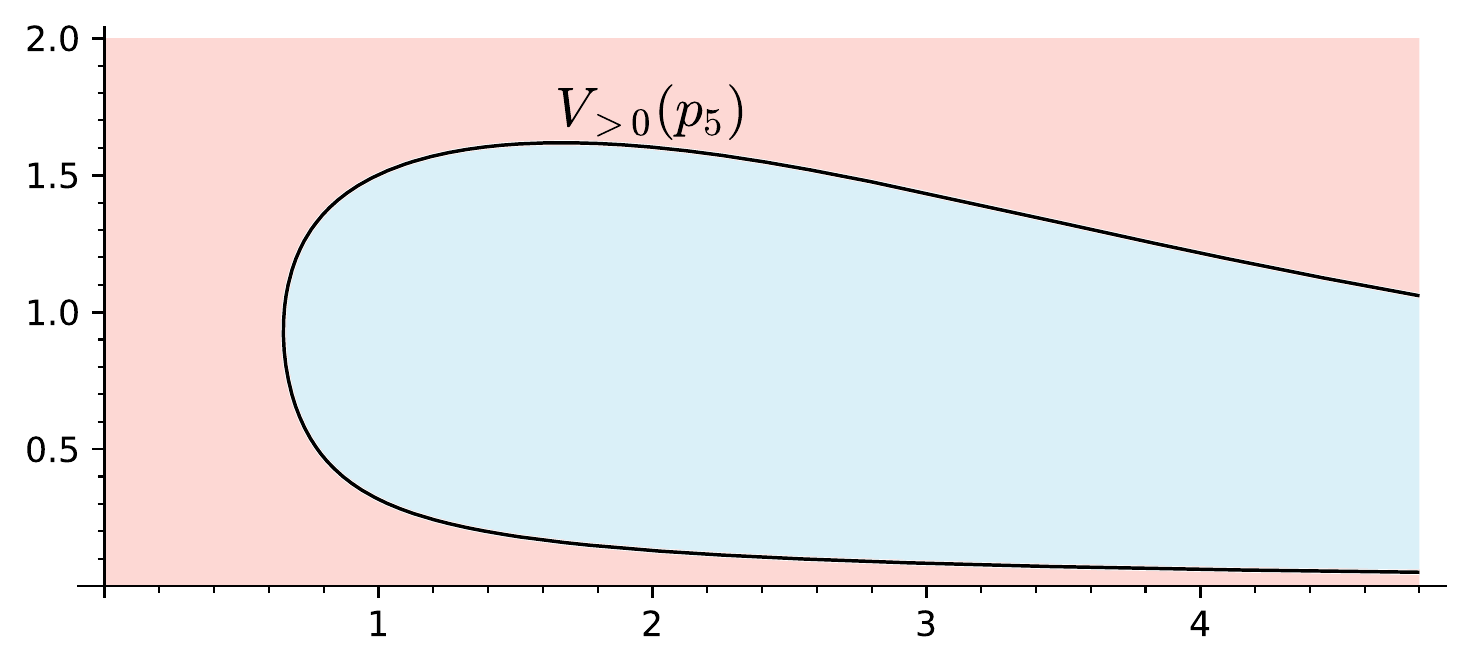}

{\small (b)}
\end{minipage}
\end{center}
\caption{{\small Illustration of Example \ref{Ex_Thm6}.  (a) Shows $\sigma(p_5)$ with blue indicating negative points and red positive points. The vector $w_{1} = (1,-1)$ is a non-strict separating vector of the support of $p_{5}$. (b) The negative connected component of $V_{>0}^{c}(p_{5})$ is shown in blue. 
}}\label{FIG2}
\end{figure} 

\begin{ex}
\label{Ex_Thm6}
Consider the signomial 
\[ p_{5}(x_{1},x_{2}) = x_{1}^{4}x_{2}^{4} + x_{1}^{2}x_{2}^{6}+ x^2 y^3 - 5x_{1}^{3}x_{2}^{3} - 3x_{1}^{2}x_{2}^{2} + x_{1}x_{2} + x_{2}^{2},\]
with $\sigma(p_5)$ depicted in Fig.~\ref{FIG2}(a). 
The vector $w_{1} = (1,-1)$ is a separating vector of $\sigma(p_5)$. The convex hull of $\sigma_{-}(p_{5})$ does not intersect $\sigma_{+}(p_{5})$ as we can see from Fig.~\ref{FIG2}(a). Hence, we can use Proposition~\ref{Lemma_n-1nonstrict} to conclude that there exists a simplex $P$ such that $\sigma_{-}(p_{5}) \subset P$ and $\sigma_{+}(p_{5}) \subset P^{-}$. In fact, the proof Proposition~\ref{Lemma_n-1nonstrict} is constructive, the corresponding $P$ is depicted also in Fig.~\ref{FIG2}(a). Now, we can apply Theorem \ref{Thm_Convexification} to conclude that $f^{-1}(\mathbb{R}_{<0})$ is contractible.
\end{ex}

\section{Comparing the different approaches}
\label{Section_Discussion}

Theorems  \ref{Thm_OneExponent}, \ref{Thm_StrictSepConnected}, \ref{Thm_StripAroundPosExp}, \ref{Thm_Convexification} cover some  cases of a generalization of Descartes' rule of signs to hypersurfaces. 
In particular, we have shown that $f^{-1}(\R_{<0})$ is   contractible in the following relevant cases:
\begin{itemize}
\item $f$ has at most one negative point in $\sigma(f)$.
\item There exists a strict separating vector of $\sigma(f)$.
\item There exists a simplex $P$ such that negative points of $\sigma(f)$ belong to $P$ and positive points to $P^{-}$; in particular if all positive points are vertices of the Newton polytope and this is a simplex, or if there are $n-1$ linearly independent non-strict separating vectors and the convex hull of the negative points does not contain positive points.
\end{itemize}

The techniques to study the case where $f^{-1}(\R_{<0})$ is path connected could also be used to derive a condition for $f^{-1}(\R_{<0})$ having at most two connected components:
\begin{itemize}
\item  There exists a strict enclosing vector of $\sigma(-f)$; in particular if the positive points belong to a hyperplane that does not contain all negative points, or if the number of positive points is smaller than $\dim \N(f)$. 
\end{itemize}

Theorem \ref{Thm_Convexification} covers all the cases where the classical Descartes' rule guarantees that the number of negative connected components of $V^{c}_{>0}(f)$ is at most one. These are the cases when the coefficients of the one-variable signomial $f$ has one of the following sign patterns:
\begin{align*}
(-\dots -&+ \dots +) &
(+ \dots +&-\dots - ) &
(+ \dots +-&\dots -+ \dots +) \, .
\end{align*}

\medskip
Although Theorem \ref{Thm_StrictSepConnected},  and \ref{Thm_Convexification} build apparently   on different techniques, we show in this section that they are equivalent in some situations.  Computationally, checking whether Theorem~\ref{Thm_StrictSepConnected} applies is less demanding  than to verifying that the conditions of Theorem~\ref{Thm_Convexification} hold.
 
\medskip

We start by noting that Theorem~\ref{Thm_Convexification} applies for the signomial $p_{4}$ in Example~\ref{Ex_Simplex},  but $\sigma(p_{4})$ does not have any separating vector.
However, under some assumptions,  the existence of an $n$-simplex as in Theorem \ref{Thm_Convexification} implies the existence of a  separating vector.
 
\begin{prop}
Let $f\colon \mathbb{R}^{n}_{>0} \to \mathbb{R}$ be a signomial and let $P \subseteq \mathbb{R}^{n}$ be an $n$-simplex such that $\sigma_{-}(f) \subseteq P$ and $\sigma_{+}(f) \subseteq P^{-}$. If there exists $k \in \{0,\dots,n\}$ such that  $ P^{-,k}\cap \sigma_{+}(f)=\emptyset$,  then $\sigma(f)$ has a separating vector. Moreover, there is a strict separating vector if there is a negative point in $P\setminus F_{k}$, where $F_{k}$ denotes the facet of $P$ opposite to $P^{-,k}$.
\end{prop}

\begin{proof}
Let $ \mathcal{H}_{v_{k},a_{k}}$ be a supporting hyperplane for  the facet $F_{k}$. 
By hypothesis and from Proposition \ref{Prop_GeomInt_Pout} we obtain $\sigma_{+}(f) \subseteq \mathcal{H}^{+}_{v_{k},a_{k}}$. 
 By hypothesis we also have that 
  $\sigma_{-}(f) \subseteq P \subseteq \mathcal{H}_{v_{k},a_{k}}^{-}$. 
  Therefore, $-v_{k}$ is a separating vector of $\sigma(f)$. 
 If there is a negative point $\beta \notin F_{k}$, then $v_{k} \cdot \beta < a_{k}$ giving that  $-v_{k}$ is strict.
\end{proof}

We inspect now whether or when Theorem  \ref{Thm_StrictSepConnected} follows from  Theorem \ref{Thm_Convexification}, in which case we obtain the additional information that $f$ can be transformed into a convex signomial.
The existence of a strict separating vector does not imply the existence of an $n$-simplex satisfying the condition in Theorem \ref{Thm_Convexification}. To see this, we consider the signomial $p_{2}$ in Example \ref{Ex_Thm10}. The positive point $(3,4)$ lies in $\Conv( \sigma_{-}(p_{2}))$, and is not a vertex. Therefore, there is no $n$-simplex $P$ such that $\sigma_{-}(p_{2}) \subseteq P$ and $(3,4) \in P^{-}$.

However, if there exists a very strict separating vector, then   there is an $n$-simplex satisfying the conditions in Theorem~\ref{Thm_Convexification} and Theorem~\ref{Thm_StrictSepConnected} follows from it. For an example, see Fig.~\ref{FIG6}.

\begin{figure}[t]
\centering
\includegraphics[width=0.4\textwidth]{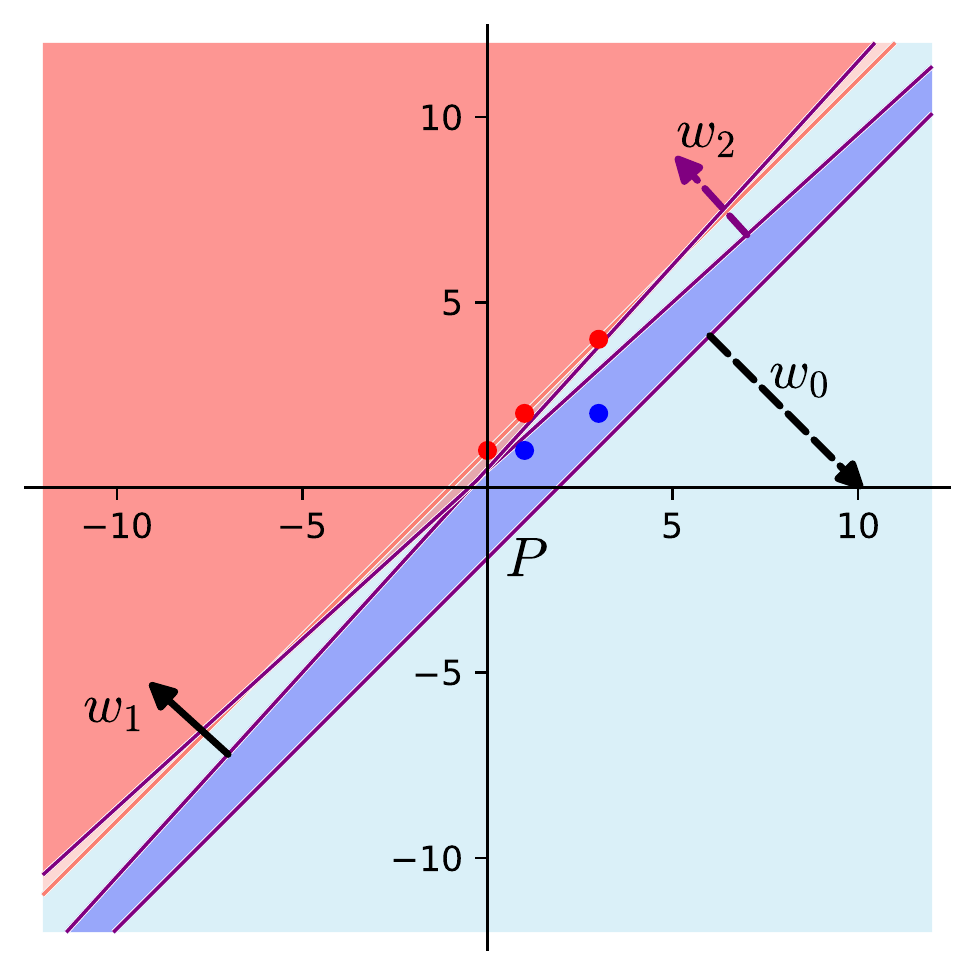}
\caption{{\small The support of the signomial  $\tilde{p}_{2}$ in Example \ref{Ex_Thm10} has a very strict  separating vector as in Proposition~\ref{Lemma_VeryStrictSepandPout}, namely $v = (1,-1)$. The $2$-simplex $P$ shown in blue is constructed following the proof of Proposition~\ref{Lemma_VeryStrictSepandPout} with the choice $v_{1}, = (1,0), v_{2} = (0,-1)$, $a_{0} = 4$.}} \label{FIG6}
\end{figure}

\begin{prop}
\label{Lemma_VeryStrictSepandPout}
Let $f\colon \mathbb{R}^{n}_{>0} \to \mathbb{R}$ be a signomial.
If there is a very strict separating vector $v \in \mathbb{R}^{n}$ of $\sigma(f)$,
then there exists an $n$-simplex $P$ such that $\sigma_{-}(f) \subseteq P$ and $\sigma_{+}(f) \subseteq P^{-}$. 
\end{prop}

\begin{proof}
By Proposition~\ref{prop_basis} 
there exist  $n$ linearly independent very strict separating vectors $-w_1,\dots,-w_n$, and $c\in \R^n$ such that
\begin{align}
\label{Align_Alpha}
\sigma_{-}(f) \subseteq \bigcap_{i=1}^{n} \mathcal{H}^{-}_{w_{i},c} \quad  \text{and} \quad \sigma_{+}(f) \subseteq \bigcap_{i=1}^{n} \mathcal{H}^{+}_{w_{i},c}. 
\end{align}
We consider minus the basis in Proposition~\ref{prop_basis}, as separating vectors leave the negative points on the positive side of the hyperplane, while the simplex $P$ leaves them on the negative side of the defining hyperplanes. 

We define $w_{0} := -\sum_{i=1}^{n}w_{i}$, choose $a_{0} \in \mathbb{R}$ such that $a_{0} > \max_{\mu \in \sigma(f)} w_{0} \cdot \mu$ and define 
\[P :=  \mathcal{H}^{-}_{w_{0},a_{0}} \cap \bigcap_{i=1}^{n} \mathcal{H}^{-}_{w_{i},c}.\]
 It then holds that  $\sigma_{-}(f)$ and $\sigma_{+}(f)$ belong to $\mathcal{H}^{-}_{w_{0},a_{0}}$. Thus, $\sigma_{-}(f) \subseteq P$, and  $\sigma_{+}(f) \subseteq P^{-}$ by Proposition~\ref{Prop_GeomInt_Pout}.

All that is left is to show that $P$ is an $n$-simplex. To this end, we apply Lemma~\ref{Lemma_HyperplaneSimplex}.
It is clear that every subset of $\{w_{0} , \dots , w_{n} \}$ with $n$ elements is linearly independent, so  Lemma~\ref{Lemma_HyperplaneSimplex}(i) holds.
From \eqref{Align_Alpha} follows that
\begin{align}
\label{eq:azero}
n\, (-c) \leq  \max\limits_{\beta \in \sigma_{-}(f)} \sum\limits_{i=1}^{n} -w_{i} \cdot \beta = \max\limits_{\beta \in \sigma_{-}(f)} w_{0} \cdot \beta  \leq  \max\limits_{\mu \in \sigma(f)} w_{0} \cdot \mu  < a_{0}. 
\end{align} 
For  $x \in \bigcap_{j = 1}^{n} \mathcal{H}_{w_{j},c}$, we obtain  $w_{0} \cdot x =  - n  \, c < a_{0}$, so $x \in \mathcal{H}^{-,\circ}_{w_{0},a_{0}}$. If $x \in   \mathcal{H}_{w_{0},a_{0}}  \cap \bigcap_{j=1, j\neq i}^{n} \mathcal{H}_{w_{j},c}$,  again by \eqref{eq:azero}   we have that 
\[ w_{i} \cdot x = - w_0 \cdot x  - \sum_{j=1,j\neq i}^{n}w_{j} \cdot x  = - a_{0} - (n-1)c <  n\, c - (n-1)\,c =c  .\] Hence $x \in \mathcal{H}^{-,\circ}_{w_{i},c}$ for each $i \in \{ 1, \dots , n \}$. We conclude that Lemma \ref{Lemma_HyperplaneSimplex}(ii) holds, so $P$ is an $n$-simplex and this completes the proof.
\end{proof}

In the scenario where $f$ has exactly one negative point neither the existence of a separating hyperplane nor the existence of a simplex satisfying the conditions of Theorem~\ref{Thm_Convexification} are guaranteed. In fact, if $f$ has one negative point, then a strict separating hyperplane exists if and only if the negative point is a vertex of the Newton polytope of $f$. The following example illustrates a scenario where a simplex as in Theorem~\ref{Thm_Convexification} does not exist, and $f$ has only one negative point.

\begin{ex}
Let $f\colon \mathbb{R}^{2}_{>0} \to \mathbb{R}$ be a signomial with only one negative point $\beta_{0}\in \sigma(f)$.  If $\sigma_{+}(f)$ is equal to the vertex set of a regular $m$-gon for some $m \geq 7$ with circumcenter $\beta_{0}$, then there does not exist a simplex $P$ such that $\sigma_{-}(f) \subseteq P$ and $\sigma_{+}(f) \subseteq P^{-}$.

 To see this, assume that such a simplex exists and write $P = \mathcal{H}^{-}_{w_{0},b_{0}}  \cap \mathcal{H}^{-}_{w_{1},b_{1}} \cap  \mathcal{H}^{-}_{w_{2},b_{2}}$, with $w_{0} ,w_{1},w_{2} \in \mathbb{R}^{2} $,   and $b_{0} ,b_{1},b_{2} \in \mathbb{R}$. For $a_i :=  w_{i}\cdot \beta_{0}$, $i=0,1,2$, the three lines $ \mathcal{H}_{w_{0},a_0}$,  $\mathcal{H}_{w_{1},a_1}$, and $\mathcal{H}_{w_{2},a_2}$,  intersect each other at $\beta_{0}$ and divide the circumsphere of the $m$-gon into $6$ regions.

Let  $\gamma_{0}, \gamma_1,\gamma_2 \in [0,\pi]$  be the angles of the regions cut out by $\mathcal{H}_{w_{0},a_0}$ and  $\mathcal{H}_{w_{1},a_1}$, by $\mathcal{H}_{w_{1},a_1}$ and $\mathcal{H}_{w_{2},a_2}$, and by $\mathcal{H}_{w_{2},a_2}$ and $ \mathcal{H}_{w_{0},a_0}$  respectively. Note that  $\gamma_{0} +  \gamma_{1} +\gamma_{2} = \pi$. Since $\sigma_{+}(f) \subseteq P^{-}$, the positive points are in alternating regions. Therefore one of the two regions cut out by $\mathcal{H}_{w_{0},a_0}$ and $\mathcal{H}_{w_{1},a_1}$  with angle $\gamma_{0}$ cannot contain any positive point. Since $\sigma_{+}(f)$ is the vertex set of a regular $m$-gon, for each pair of consecutive positive point $\alpha_{i},\alpha_{i+1}$ (counted counterclockwise), the angle $\measuredangle \alpha_{i}\beta_{0}\alpha_{i+1}$ equals $\tfrac{2\pi}{m}$. From this follows that $\gamma_{0}\leq  \tfrac{2 \pi}{m}$. A similar argument shows that $\gamma_{1}\leq  \tfrac{2 \pi}{m}$, $\gamma_{2}\leq  \tfrac{2 \pi}{m}$. We conclude that $\gamma_{0} + \gamma_{1} + \gamma_{2} \leq \tfrac{6 \pi}{m}$. Since $m \geq 7$, this contradicts   $\gamma_{0} + \gamma_{1} + \gamma_{2} = \pi$.
Therefore, such a simplex $P$ does not exist.
\end{ex}

{\small
\bibliographystyle{plain}
\bibliography{references}
}

\end{document}